%% file: Arxiv.tex
\pdfoutput=1

\documentclass[11pt]{article}

\usepackage[longnamesfirst]{natbib} 
\usepackage{amsmath} 
\usepackage{amsfonts} 
\usepackage{amsthm}
\usepackage{amssymb}
\usepackage{graphicx} 
\usepackage{xcolor} 
\usepackage{xspace}
\usepackage{ifmtarg} 

\RequirePackage[colorlinks,citecolor=black,urlcolor=black]{hyperref} 
\hypersetup{pdfauthor={J.-P. Baudry},pdftitle={Estimation and Model Selection for Model-Based Clustering with the Conditional Classification Likelihood}}

\input{aliases}

\title{Estimation and Model Selection for Model-Based Clustering with the Conditional Classification Likelihood}

\author{Jean-Patrick Baudry\footnote{This research was partly supported by Universit\'e Paris XI, Laboratoire de Math\'ematiques d'Orsay UMR 8628; INRIA Saclay \^Ile-de-France, Projet {\sc select} and Universit\'e Paris Descartes, MAP5 UMR 8145.} \\ \emph{LSTA} \\ \emph{Universit\'e Pierre et Marie Curie - Paris VI}}

\begin{document}

\maketitle

\begin{abstract}

The Integrated Completed Likelihood (ICL) criterion has been proposed by \citet{BieCelGov00} in the model-based clustering framework to select a relevant number of classes and has been used by statisticians in various application areas. 

A theoretical study of this criterion is proposed. A contrast related to the clustering objective is introduced: the \emph{conditional classification likelihood}. This yields an estimator and a model selection criteria class. The properties of these new procedures are studied and ICL is proved to be an approximation of one of these criteria. 

We oppose these results to the current leading point of view about ICL, that it would not be consistent. Moreover these results give insights into the class notion underlying ICL and feed a reflection on the class notion in clustering.

General results on penalized minimum contrast criteria and on mixture models are derived, which are interesting in their own right.

\end{abstract}


\section{Introduction}\label{icl.int}

Model-based clustering is introduced in Sections \ref{icl.int.gmm} and \ref{icl.int.mbc}. Our purpose is to better understand the behavior of the ICL model selection criterion of \citet{BieCelGov00}, which is presented in Section \ref{icl.int.icl}. 

The main topic of this work is the choice of the number of classes in a model-based clustering framework, and then the choice of the number of components of a Gaussian mixture. The interested reader may refer to \citet{TitSmiMak85} or \citet{MclPee00} for comprehensive studies on Gaussian mixture models. The last also provides an overview on the approaches for assessing the number of components, and particularly on the standard and widely used penalized likelihood criteria, such as AIC \citep{Aka73} or BIC \citep{Sch78}. 

The ICL criterion studied here is an alternative to BIC. It was up to now widely presented as a penalized likelihood criterion, which penalty involves an ``entropy'' term. Here, however, we prove that it is actually a penalized contrast criterion with a criterion which is different from the standard likelihood: this justifies why this is not surprising, nor a drawback, that ICL does not asymptotically select the ``true'' number of components, even when the ``true'' model is considered. Even for data arising from a mixture distribution, a relevant number of classes may differ from the true number of components of the mixture. 

The reason why we introduce this new contrast \lcc (\sub{icl.con.def}) is not that we believe it \emph{a priori} to be the better one for a clustering purpose, but rather that it enables to theoretically study and understand ICL.  We prove (\sub{icl.sel.icl}) that ICL is an approximation of a criterion linked to this contrast: studying further ICL then amounts to studying \lcc. The notion of class underlying ICL is proved to be a compromise between Gaussian mixture density estimation and a strictly ``cluster'' point of view (\sect{icl.dis}).

Let $X$ be a \rv in \Rd with distribution $f^\wp\cdot\lambda$ and $X_1,\dots,X_n$ an i.i.d. sample of the same distribution. Let us denote ${\bf X} = (X_1,\dots,X_n)$.

All proofs are gathered in \sect{icl.pro}.

\subsection{Gaussian Mixture Models}\label{icl.int.gmm}
$\mathcal{M}_K$ is the Gaussian mixture model with $K$ components:
\begin{equation*}
\mathcal{M}_K = \left\{ f(\p\theta) = \sum_{k=1}^K \prop_k \phi(\p \omega_k)\lo \theta = (\prop_1,\dots,\prop_K,\omega_1,\dots,\omega_K)\in \Theta_K \right\},
\end{equation*}
where $\phi$ is the Gaussian density and $\Theta_K \subset \Prop_K \times \left(\Rd\times\Sd\right)^K$ with $\Prop_K = \{(\prop_1,\dots,\prop_K) \in [0,1]^K : \sum_{k=1}^K \prop_k = 1 \}$ and $\Sd$ the set of positive definite $d\times d$ real matrices. Constraints on the model may be imposed by restricting $\Theta_K$. We typically have in mind the decomposition suggested by \cite{CelGov95}. ``General'' (no constraint) and ``diagonal'' (diagonal covariance matrices) models will be considered here as examples.

Those are studied here as parametric models. It is then assumed the existence of a parametrization $\varphi:\Theta_K\subset\mathbb{R}^{D_K}\rightarrow\mathcal{M}_K$. It is assumed that $\Theta_K$ and $\varphi$ are ``optimal'', in the sense that $D_K$ is minimal. $D_K$ is the number of \emph{free parameters} in the model $\mathcal{M}_K$ and is called the \emph{dimension} of $\mathcal{M_K}$. For example, at most $(K-1)$ mixing proportions need to be parametrized. 

It shall not be needed to assume the parametrization to be identifiable, \ie that $\varphi$ is injective. Indeed our purpose is twofold: identifying a relevant number of classes to be designed; and actually designing those classes. Theorem \ref{th.icl.selgmm} justifies that the first task can be achieved under a weaker ``identifiability'' assumption. Theorem \ref{th.icl.conestgmm} then guarantees that our estimator converges to the best parameters set, any of which is as good as the others. There will be no ``true parameter'' assumption. The classes can finally be defined through the MAP rule (see \sub{icl.int.mbc}). Practically, the parameters themselves are never the quantities of interest here. They only stand as a convenient notation and this is also why we expect that the assumption about the Fisher information (see Theorem~\ref{th.icl.selgmm}) is technical and could maybe be avoided with other techniques. Please refer to \citet[Chapter 4]{Bau09} for a more comprehensive discussion about the identifiability question.

\subsection{Model-Based Clustering}\label{icl.int.mbc}

Although the results are stated first for much more general situations, this paper is devoted to the question of clustering through Gaussian mixture models.

The process is standard \citep[see][]{FraRaf02}: 
\begin{itemize}
\item fit each considered mixture model;
\item select a model and a number of components based on the first step;
\item classify the observations through the MAP rule (recalled below) with respect to the mixture distribution fitted in the selected model.
\end{itemize}
Notably, the usual choice is made here, to identify a class with each fitted Gaussian component. The number of classes to be designed is then chosen at the second step. See for example \citet{Hen10} and \citet{BauCelRafal09} for alternative approaches.

Let us recall the MAP classification rule. It involves the \emph{conditional probabilities} of the components
\begin{equation*}
\forall \theta\in\Theta_K, \forall k, \forall x,\ \tau_k(x;\theta) = \frac{\prop_k \phi(x;\omega_k)}{\sum_{k'=1}^K \prop_{k'} \phi(x;\omega_{k'})}\cdot
\end{equation*}
$\tau_k(x;\theta)$ is the probability that $X$ arises from the $k^\text{th}$ component, conditionally to $X=x$, under the distribution defined by $\theta$. Let us also denote $\tau_{ik}(\theta) = \tau_k(X_i;\theta)$. The MAP classification rule for $x$ is then 
\begin{equation*}
\MAP{z}(\theta) = \argmax_{k\in\{1,\dots,K\}} \tau_k(x;\theta).
\end{equation*}
Let us denote by \Lo the \emph{observed likelihood} associated to ${\bf X}$:
\begin{equation*}
\forall\theta\in\Theta_K,\ \Lo(\theta;{\bf X}) = \prod_{i=1}^n \Bigl(\sum_{k=1}^K \prop_k\phi(X_i;\omega_k)\Bigr).
\end{equation*}
The maximum likelihood estimator in the model $\mathcal{M}_K$ is denoted by $\MLE\theta_K$.

\subsection{ICL}\label{icl.int.icl}

Our motivation is to better understand the ICL (Integrated Completed Likelihood) criterion. Let us introduce the \emph{classification likelihood} associated to the complete data sample $(\bf X,\bf Z)$ ($Z\in\{0,1\}^K$ is the unobserved label of $X$: $Z_k=1\Leftrightarrow$ $X$ arises from component $k$):
\begin{equation}\label{eq.Lc}
\forall\theta\in\Theta_K,\ \Lc\bigl(\theta;({\bf X},{\bf Z})\bigr)=\prod_{i=1}^n \prod_{k=1}^K \bigl(\prop_k \phi(X_i;\omega_k) \bigr)^{Z_{ik}}.
\end{equation}

To mimic the derivation of the BIC criterion \citep{Sch78} in a clustering framework, \cite{BieCelGov00} approximate the integrated classification likelihood through a Laplace's approximation. Then they assume that the classification likelihood mode can be identified with $\MLE\theta_K$ as $n$ is large enough and replace the unobserved $Z_{ik}$'s by their MAP estimators under $\MLE\theta_K$. This is questionable, notably when the components of $\MLE\theta_K$ are not well separated. They derive the ICL criterion:
\begin{equation*}
\crit_{\ICL} (K) = \log \Lo(\MLE\theta_K) + \sum_{i=1}^n \sum_{k=1}^K \MAP{Z}_{ik}(\MLE\theta_K) \log \tau_{ik}(\MLE\theta_K) - \frac{\log n}{2} D_K. 
\end{equation*}
\cite{MclPee00} replace the $Z_{ik}$'s by their conditional expectations $\tau_{ik}(\MLE\theta_K)$:
\begin{equation}\label{eq.icl.deficl}
\crit_{\ICL} (K) = \log \Lo(\MLE\theta_K) + \sum_{i=1}^n \sum_{k=1}^K \tau_{ik}(\MLE\theta_K) \log \tau_{ik}(\MLE\theta_K) - \frac{\log n}{2} D_K. 
\end{equation}
Both versions of the ICL appear to behave analogously, and the latter is considered from now on.

The ICL differs from the standard and widely used BIC criterion of \cite{Sch78} through the \emph{entropy} term (see \sub{icl.con.ent}):
\begin{equation}\label{eq.defEnt}
\forall \theta\in\Theta_K,\ \ent(\theta;{\bf X}) = - \sum_{i=1}^n \sum_{k=1}^K \tau_{ik}(\theta) \log \tau_{ik}(\theta).
\end{equation}

The BIC is known to be consistent, in the sense that it asymptotically selects the true number of components, at least when the true distribution actually lies in one of the considered models \citep{Ker00, Nis88}. This nice property may however not suit a clustering purpose. In many applications, there is no reason to assume that the distribution conditional on the (unobserved) labels $Z$ is Gaussian. The BIC in this case tends to overestimate the number of components since several Gaussian components are needed to approximate each non-Gaussian component of the true mixture distribution $f^\wp$. And the user may rather be interested in a \emph{cluster} notion ---~as opposed to this strictly \emph{component} approach~--- which also includes a separation notion and which be robust to non-Gaussian components. Of course, it depends on the application, and on what a \emph{class} should be. It may be of interest to discriminate into two different classes a group of observations which the best fit is reached with a mixture of two Gaussian components having quite different parameters (we particularly think of the covariance matrices parameters). BIC tends to do so. But it may also be more relevant and may conform to an intuitive notion of cluster, to identify two very close ---~or largely overlapping~--- Gaussian components as a single non-Gaussian shaped cluster (see for example \fig{fig.icl.ent})...

ICL has been derived with this viewpoint. It is widely understood and explained \citep[for instance in][]{BieCelGov00} as the BIC criterion with a supplemental penalty, which is the entropy (\sub{icl.con.ent}). Since the last penalizes models which maximum likelihood estimator yields an uncertain MAP classification, ICL is more robust than BIC to non-Gaussian components. However we do not think that the entropy should be considered as a penalty term and an other point of view will be developed in this paper.

The references here were found by browsing the result obtained from Google Scholar citations about \cite{BieCelGov00}. Only 3 pages of 16 have been studied...
The behavior of ICL has been studied through simulations and real data studies by \citet{BieCelGov00}, \citet[Section 6.11]{MclPee00}, \citet{SteRaf10} and in several simulation studies \citep[See][Chapter 4]{Bau09}. Besides several authors chose to use it for the mentioned reasons in various applications area: \cite{GouHansLipRos01} (fMRI images); \cite{PigGel05} (image collection automatic sorting); \cite{HamKenKro06} (protein structure prediction); \cite{GraSouFag06} (robots learning); \citet{MarRobVac10} (uncovering groups of nodes in valued graphs and application to host-parasite interaction networks in forest ecosystems analysis); \citet{RigLebRob12} (comparative genomic hybridization profile); etc. 

This practical interest for ICL lets us think that it meets an interesting notion of cluster, corresponding to what some users expect. But no theoretical study is available. Our main motivation is to go further in this direction. This leads to considering new estimation and model selection procedures for clustering, similar to ICL but for which the development of the underlying logic is driven to its conclusion, from the estimation step to the model selection step, instead of introducing the MLE. It is proved that ICL is an approximation of a criterion which is consistent for a particular loss function.

\section{A New Contrast: Conditional Classification Likelihood}\label{icl.con}

The contrast minimization framework turns out to be a fruitful approach. It enables to fully understand that ICL is not a penalized likelihood criterion, as opposed to the usual point of view. It should rather be linked to an other contrast: the \emph{conditional classification likelihood}. 

\subsection{Definition, Origin}\label{icl.con.def}

In a clustering context, the classification likelihood (see \eqref{eq.Lc}) is an interesting quantity but neither the labels $\bf Z$ are observed, nor we assume that they even exist (think of the case several models with different number of components are fitted: then at most one can correspond to the true number of classes). Beside the first-mentioned works of \citet{BieCelGov00}, \citet{BieGov97}, for example, already proposed to directly involve the classification likelihood to select the number of classes, by estimating the unobserved data. We propose here to consider its expectation conditional on the observed sample $\bf X$. In case there exists a true classification and a model with the true number of classes is considered, this conditional expectation can be interpreted as the quantity the closest to the classification likelihood, which can be considered given the available information. 

Let us report the following algebraic relation between \Lo and \Lc:
\begin{equation}\label{eq.fund}
\forall\theta\in\Theta_K,\ \log \Lc(\theta) = \log \Lo(\theta) + \sum_{i=1}^n\sum_{k=1}^K Z_{ik}\log\tau_{ik}(\theta).
\end{equation}
Then, denoting the conditional expectation of $\log\Lc(\theta)$ by $\Lcc(\theta)$ (for Conditional Classification log Likelihood),
\begin{align*}
\Lcc(\theta) &= \E{\theta} {\log \Lc(\theta)|{\bf X}}\\
&= \log \Lo(\theta) + \underbrace{\sum_{i=1}^n \sum_{k=1}^K \tau_{ik}(\theta)\log\tau_{ik}(\theta)}_{-\ent(\theta;{\bf  X})},
\end{align*}
which is obviously linked to the clustering objective. We consider in the following $-\Lcc$ as an empirical contrast to be minimized.

\subsection{Entropy}\label{icl.con.ent}

\Lcc differs from $\log \Lo$ through the \emph{entropy} (see \eqref{eq.defEnt}). 

The behavior of the entropy is based on the properties of the function $h : t\in [0,1] \longmapsto (-t\log t)$ (with $h(0)=0$).
\begin{figure}[ht]
\begin{minipage}[c]{0.4\linewidth}
\vspace{1cm}
	\begin{center}
		\vspace{-2.5em}
		\includegraphics[width=.95\linewidth,height=.95\linewidth]{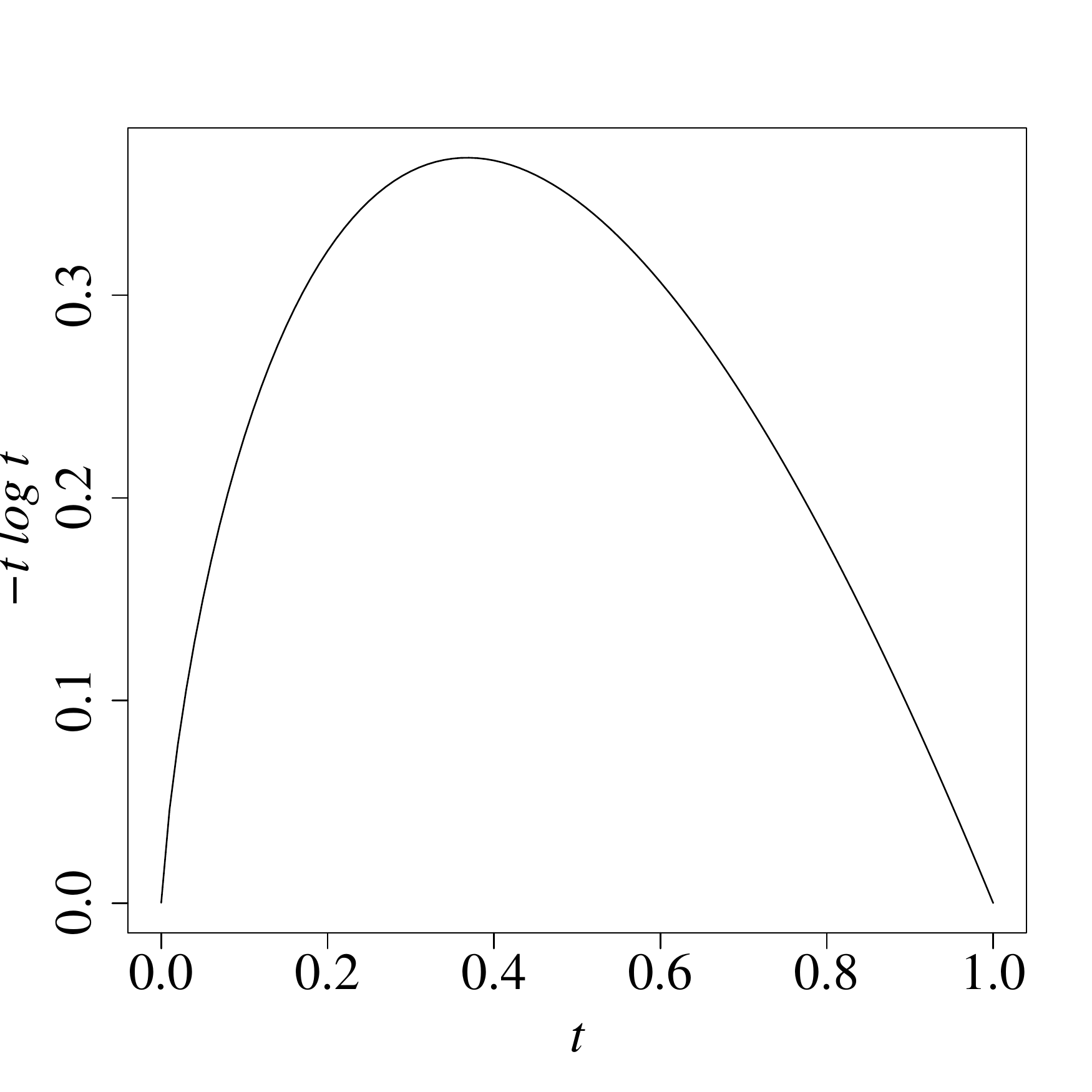}
		\vspace{-.5em}
		\caption{$h : t \mapsto -t\log t$}\label{fig.icl.h}
	\end{center}
\end{minipage}
\hspace{0.04\linewidth}
\begin{minipage}[c]{0.53\linewidth}
	\begin{center}
		\vspace{-1em}
		\includegraphics[width=.8\linewidth,angle=-90,viewport=0cm 1.5cm 21cm 29cm,clip]{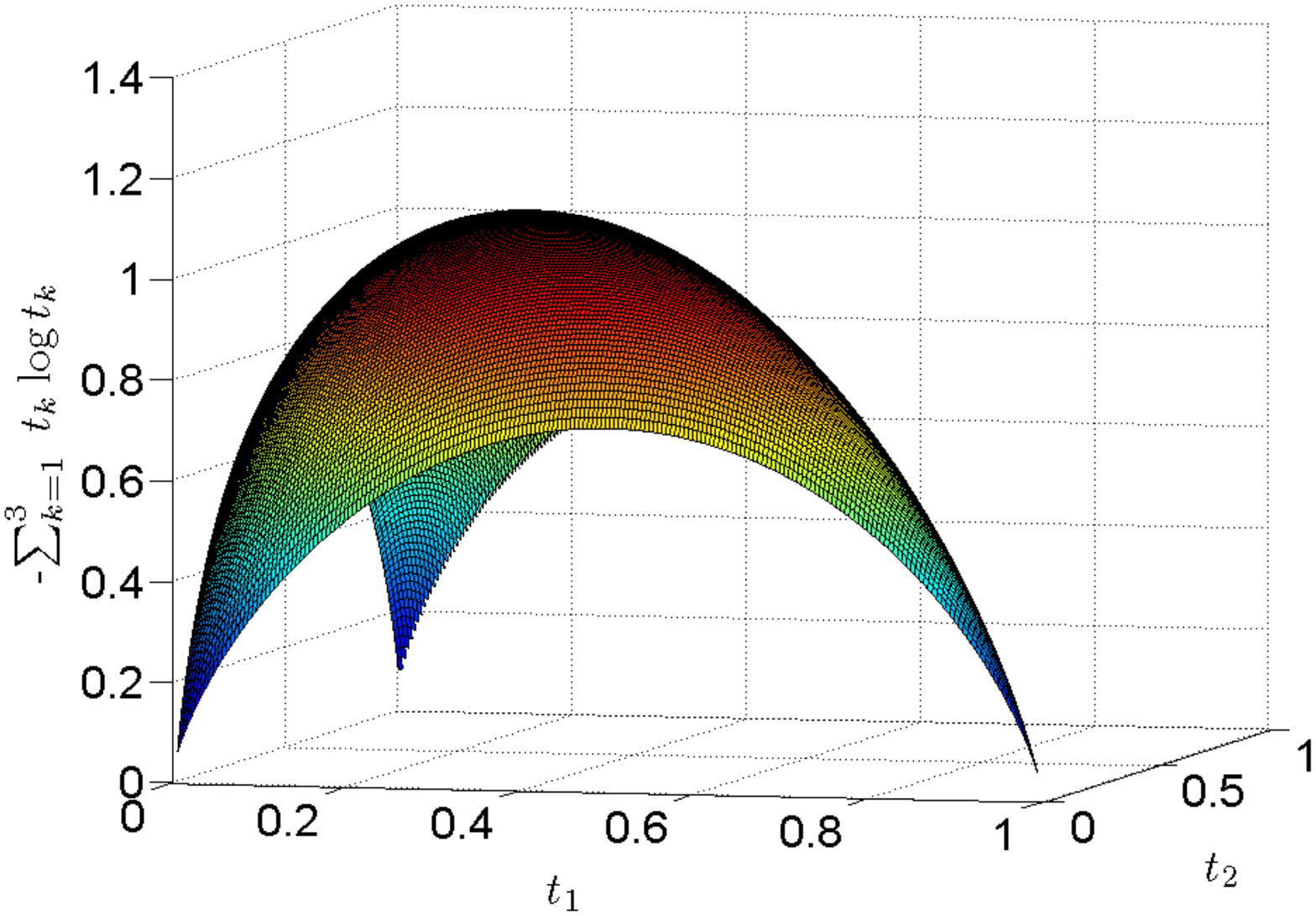}
		\vspace{-2em}
		\caption{$h_3 : (t_1,t_2,t_3) \in \Pi_3 \mapsto \sum_{k=1}^3 h(t_k)$}\label{fig.icl.hk}
	\end{center}
\end{minipage}
\end{figure}
This nonnegative function (see \fig{fig.icl.h}) takes zero value if and only if $t=0$ or $t=1$. It is continuous but not differentiable at $0$, and in particular it is not Lipschitz over $[0,1]$, which will be a cause of analysis difficulties. 
Let us also introduce the function $h_K : (t_1,\dots,t_K) \in \Prop_K \longmapsto \sum_{k=1}^K h(t_k).$
This nonnegative function (see \fig{fig.icl.hk}) then takes zero value if and only if there exists $k_0\in\{1,\dots,K\}$ such that $t_{k_0}=1$ and $t_k=0$ for $k\neq k_0$. It reaches its maximum value $\log K$ at $(t_1,\dots,t_K)=(\inv{K},\dots,\inv{K})$ (proof in \sect{icl.pro}).\label{hk} 

Now, the contribution $\ent(\theta;x_i)$ of a single observation to the total entropy $\ent(\theta;{\bf x})$ is considered. \fig{fig.icl.ent} represents a dataset simulated from a four-component Gaussian mixture. Let $\theta$ be such that $f(\p\theta)=f^\wp$. First, $\ent(\theta;x_i) \approx 0$ if and only if there exists $k_0$ such that $\tau_{i{k_0}} \approx 1$ and $\tau_{i{k}} \approx 0$ for $k\neq k_0$. There is no difficulty to classify $x_i$ in such a case (for example $x_{i_1}$). Second, $\ent(\theta;x_i)$ is all the greater that $(\tau_{i1},\dots,\tau_{iK})$ is closer to $(\inv{K},\dots,\inv{K})$, \ie that the classification through the MAP rule is uncertain. The worst case is reached as the conditional distribution over of the components $1,\dots,K$ is uniform. The observation $x_{i_2}$ for example has about the same posterior probability $\frac{1}{2}$ to arise from each one of the components surrounding it. Its individual entropy is about $\log 2$. 

In conclusion the individual entropy is a measure of the \emph{assignment confidence} of the considered observation through the MAP classification rule. The total entropy $\ent(\theta;{\bf x})$ is the empirical mean assignment confidence, and then measures the MAP classification quality for the whole sample.
\begin{figure}[ht]
\begin{center}
\def\figwidth{0.5\linewidth}
\vspace{-1em}
\includegraphics[width=\figwidth]{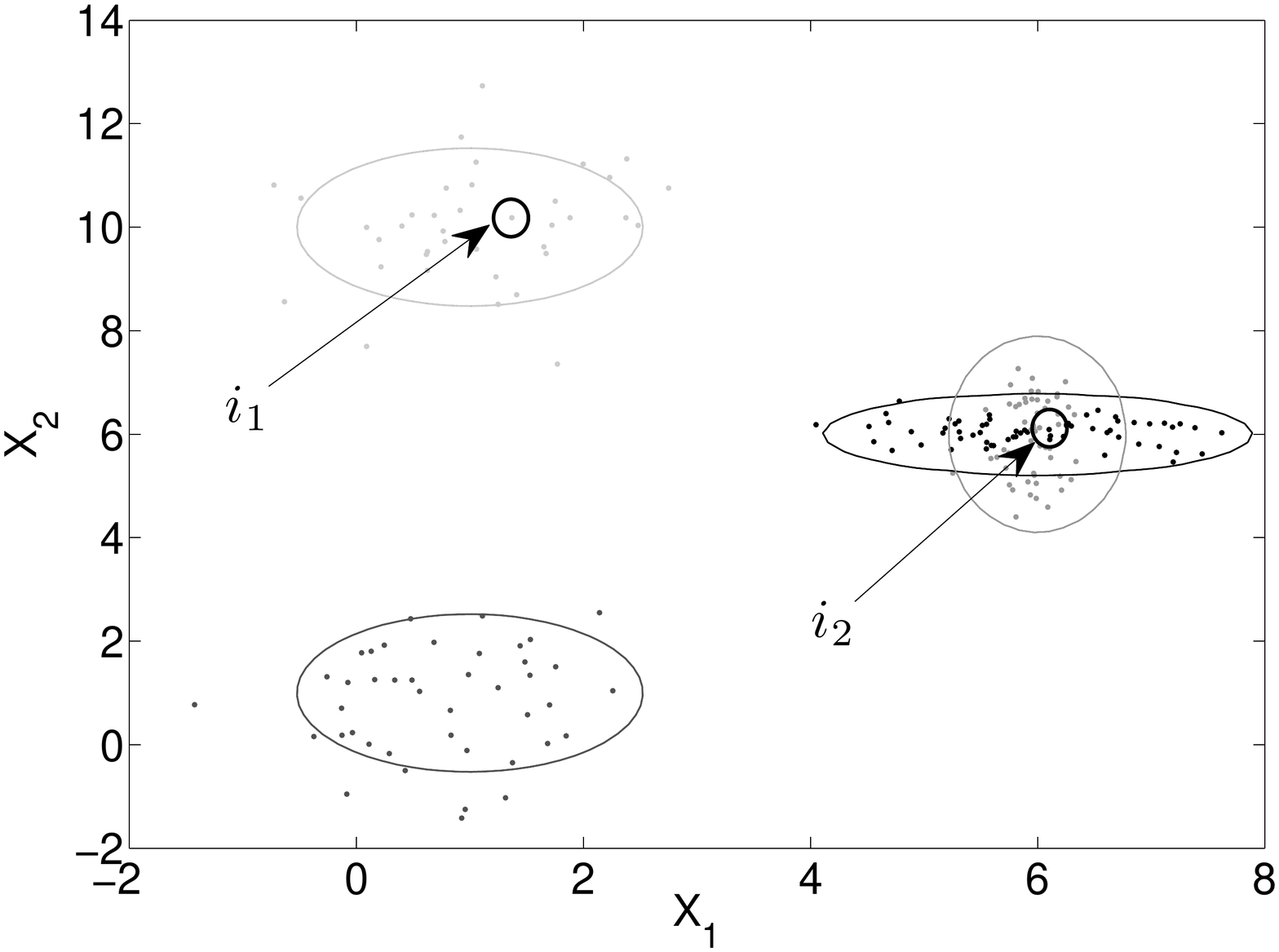}
\vspace{-1em}
\caption{A dataset example}\label{fig.icl.ent}
\vspace{-1em}
\end{center}
\end{figure}

Involving this quantity in a clustering study means that one expects the classification to be confident. The class notion underlying the choice of the conditional classification likelihood as a contrast is then a compromise between the fit (and then the idea of Gaussian-shaped classes) because of the likelihood term on the one hand, and the assignment confidence because of the entropy term on the other hand (which is rather a \emph{cluster} point of view).	

\subsection{$\Lcc$ as a Contrast}\label{icl.con.con}
See for example \citet{Mas07} for an introduction to contrast minimization. Let us consider the best distribution from the \lcc point of view in a model $\mathcal{M}_m=\{f(\p\theta):\theta\in\Theta_m\}$, namely the distribution minimizing the corresponding \emph{loss function}
\begin{equation*}
\theta_m \in \underbrace{\argmin_{\theta\in\Theta_m} \bigl\{ \KL\bigl(f^\wp,f(\p\theta)\bigr) + \E{f^\wp}{\ent(\theta;X)} \bigr\}}_{\underbrace{\argmin_{\theta\in\Theta_m} \E{f^\wp}{-\Lcc(\theta)}}_{\text{this set is denoted by }\Theta_m^0}}.
\end{equation*}
The existence of $\E{f^\wp}{-\Lcc(\theta)}$ is a very mild assumption. The non-emptiness of $\Theta_m^0$ may be guaranteed for example by assuming $\Theta_m$ to be compact. Let $K$ be fixed and consider the minimization of the loss function at hand in the model $\mathcal{M}_K$ (\sub{icl.int.gmm}). First of all, remark that $\Lcc=\log\Lo$ if $K=1$: $\Theta_1^0$ is the set of parameters of the distributions which minimize the Kullback-Leibler divergence to $f^\wp$. Now, if $K>1$, $\theta_K^0 \in \Theta_K^0$ may be close to minimizing the Kullback-Leibler divergence if the corresponding components do not overlap since then, the entropy is about zero. But if those components overlap, this is not the case anymore (Example \ref{ex.icl.ex1}).

To completely define the loss function, and to fully understand this framework, it is necessary to consider the \emph{best element of the universe} $\mathcal{U}$:
\begin{equation*}
\argmin_{\theta\in\,\mathcal{U}} \E{f^\wp}{-\Lcc(\theta)}.
\end{equation*}
The \emph{universe} $\mathcal{U}$ must be chosen with care. There is no natural relevant choice, on the contrary to the density estimation framework where the set of all densities may be chosen. First the considered contrast is well-defined in a parametric mixture setup, and not necessarily over any mixture densities set because of the definition of the entropy term involving the definition of each component. However, this would still enable to consider mixtures much more general than mixtures of Gaussian components. The ideas developed in \citet{BauCelRafal09} may for example suggest to involve mixtures which components are Gaussian mixtures. But this would not make sense. The mixture with one component which is a mixture of $K$ Gaussian components, and which then yields a single non-Gaussian-shaped class, always has a smaller $-\Lcc$ value than the corresponding Gaussian mixture yielding $K$ classes. This illustrates how carefully the components involved in the study must be chosen: involving for example any mixture of Gaussian mixtures means that one considers that a class may be almost anything and may notably contain two Gaussian-shaped clusters very far from each other! The components should in any case be chosen with respect to the corresponding cluster shape. The most natural is then to involve in the universe only Gaussian mixtures: $\mathcal{U}$ may be chosen as $\cup_{1\leq K\leq K_M}\mathcal{M}_K$.

\begin{Ex}\label{ex.icl.ex1} $f^\wp$ is the normal density $\mathcal{N}(0,1)$ ($d=1$). The model ${\mathcal{M}}_2=\{ \frac{1}{2} \phi(\p -\mu,\sigma^2) + \frac{1}{2} \phi(\p \mu,\sigma^2) ; \mu\in\mathbb{R},\ \sigma^2>0 \}$ is considered. 

Let us consider $\Theta^0_2$ in this most simple situation. 
We numerically obtain that $\Theta_2^0=\{(-\mu_0,\sigma_0^2),(\mu_0,\sigma_0^2)\}$, so that, up to a label switch, there exists a unique minimizer of $\E{f^\wp}{-\Lcc(\mu,\sigma^2)}$ in $\Theta_2$ in this case (see \fig{fig.icl.ex11}), with $\mu_0\approx 0.83$ and $\sigma_0^2\approx 0.31$. This solution is obviously not the same as the one minimizing the Kullback-Leibler divergence (see \fig{fig.icl.ex12}). This illustrates that the objective with the $-\Lcc$ contrast is not to recover the true distribution, even when it is available in the considered model. 

The necessity of choosing a relevant model is striking in this example: this two-component model should obviously not be used for a clustering purpose, at least for datasets with great enough size.

\begin{figure}[ht]
\begin{minipage}[c]{0.47\linewidth}
	\begin{center}
	    \includegraphics[width=\linewidth,height=0.9\linewidth]{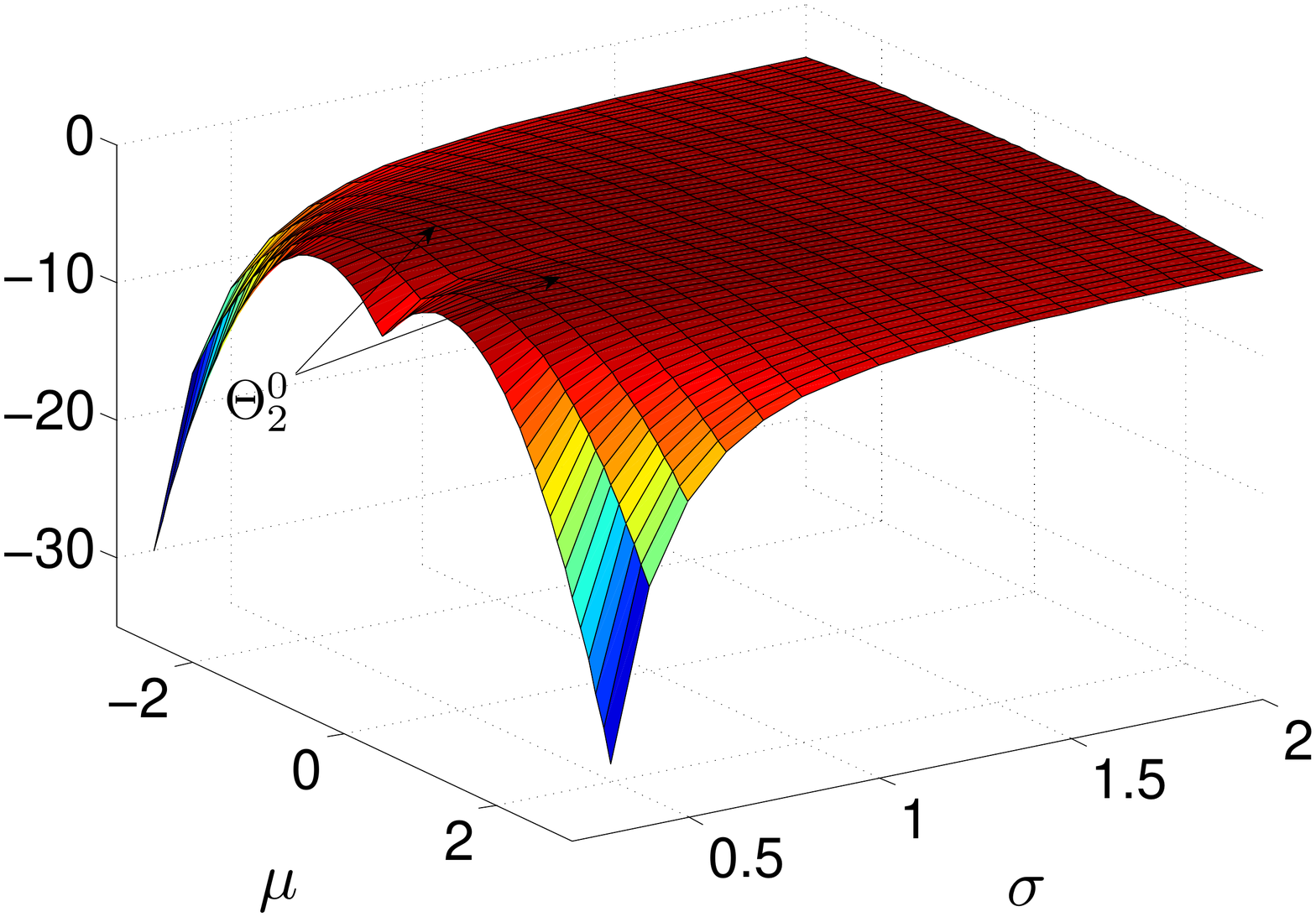}
		\vspace{-1em}
		\caption{$\E{f^\wp}{\Lcc(\mu,\sigma^2)}$ w.r.t. $\mu$ and $\sigma$, and $\Theta_K^0$, for Example \ref{ex.icl.ex1}}
		\label{fig.icl.ex11}
	\end{center}
\end{minipage}
\hspace{0.04\linewidth}
\begin{minipage}[c]{0.47\linewidth}
	\begin{center}
   \includegraphics[width=1\linewidth,height=0.8\linewidth,viewport=0.9cm 0cm 29.7cm 21cm, clip]{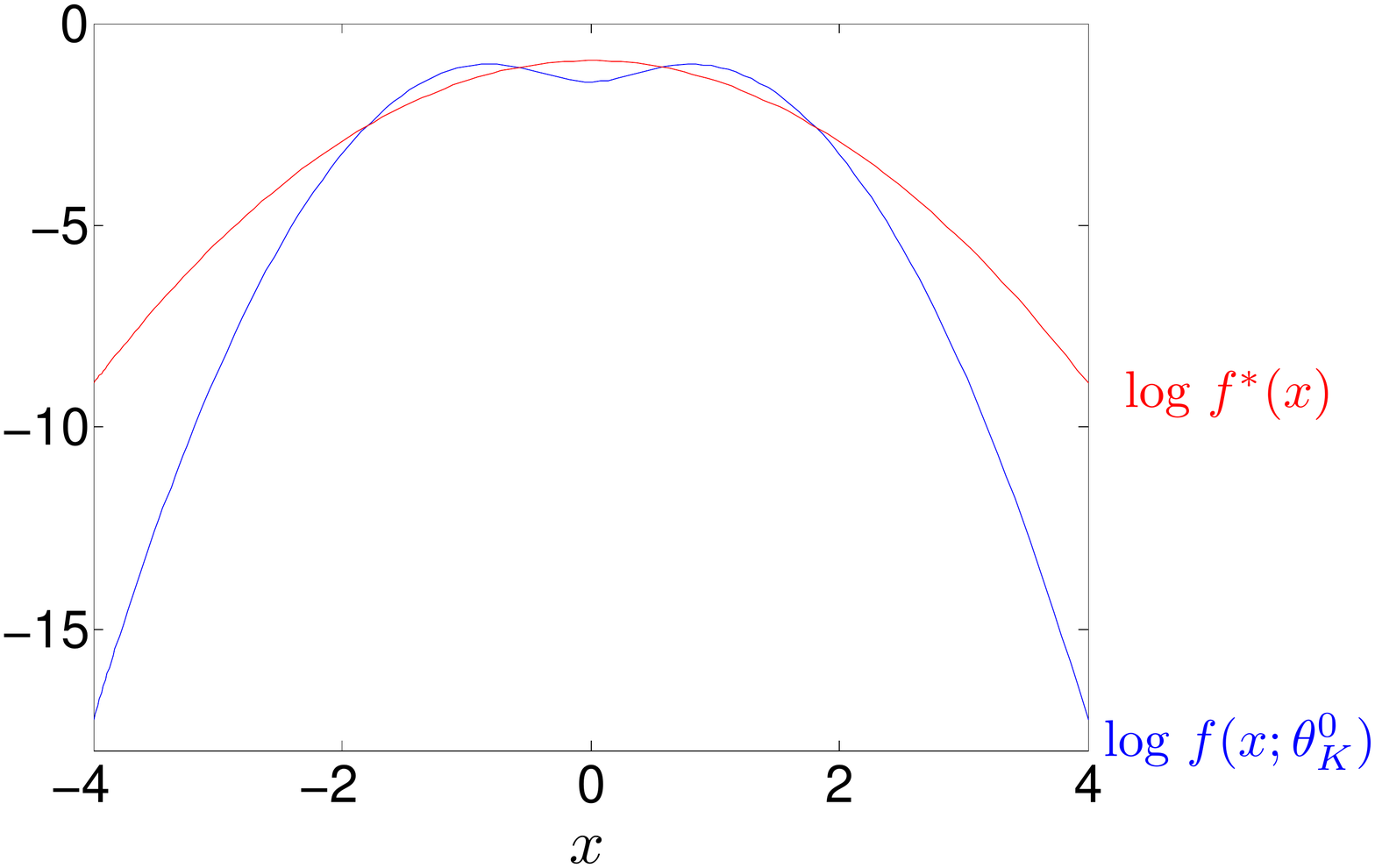}
		\vspace{-1em}
		\caption{$\log f^\wp$ (red, which is also $\log f(\p \theta_2^\text{KL})$) and $\log f(\p \theta_2^0)$ (blue) for Example \ref{ex.icl.ex1}}
		\label{fig.icl.ex12}
	\end{center}
\end{minipage}
\end{figure}
\vspace{-1em}
\end{Ex}

The estimator associated to the $-\Lcc$ contrast is now considered.

\section{Estimation: MLccE}\label{icl.est}

Let us fix the number of components $K$ and the model $\mathcal{M}_K$. The subscript $K$ is omitted in the notation of this section. A new minimum contrast estimator is considered. Results are stated in a general parametric model setting with a general contrast $\gamma$ and a model $\mathcal{M}$ with parameter space $\Theta\subset\mathbb{R}^D$, and then the conditions they involve are discussed in our framework. General conditions ensuring the consistency of such an estimator are given in Theorem \ref{th.icl.conest}. They notably involve the Glivenko-Cantelli property of the class of functions $\{\gamma(\theta):\theta\in\Theta\}$. Sufficient conditions in terms of bracketing entropy for this property to hold are recalled and verified in the considered context in \sub{icl.con.bra}. Those results are also useful in the study of the model selection step (\sect{icl.sel}). Brought together, they provide the consistency of the estimator in Gaussian mixture models: this is Theorem \ref{th.icl.conestgmm}. 

Here and hereafter, all expectations $\mathbb{E}$ and probabilities $\mathbb{P}$ are taken with respect to $f^\wp\cdot\lambda$. For a general contrast $\gamma$, we write its empirical version: $\gamma_n (\theta) = \inv{n} \sum_{i=1}^n \gamma(\theta;X_i)$. $\mathbb{R}^D$ is equipped with the infinite norm: $\forall\theta\in\mathbb{R}^D,\|\theta\|_\infty=\max_{1\leq i \leq D} |\theta_i|$. 
For any $r\in\mathbb{N}^*\cup\{\infty\}$ and for any $g:\mathbb{R}^d\rightarrow \mathbb{R}$, $\|g\|_r = \E{f^\wp}{|g(X)|^r}^\inv{r}$ if $r<\infty$ and $\|g\|_\infty = \esssup_{X\sim f^\wp}{|g(X)|}$ (recall that $\esssup_{Z\sim\mathbb{P}} Z = \inf\{z : \mathbb{P}[Z\leq z]=1 \}$ and thus: $\|g\|_\infty \leq \sup_{x\in\supp f^\wp} |g(x)|$). For any linear form $l: \mathbb{R}^D\rightarrow\mathbb{R}$, $\|l\|_\infty  = \underset{\|\theta\|_\infty=1}{\max} l(\theta)$. 

\subsection{Definition, Consistency}\label{icl.est.def}

The minimum contrast estimator is named MLccE (Maximum conditional classification Likelihood Estimator):
\begin{equation*}
\MLccE{\theta} \in \argmin_{\theta\in\Theta} -\Lcc(\theta).
\end{equation*}
To ensure its existence, we assume that $\Theta$ is compact. This is a heavy assumption, but it will be natural and necessary for the following results to hold. That the covariance matrices are bounded from below is a reasonable and necessary assumption in the Gaussian mixture framework: without this assumption, neither the log likelihood, nor the conditional classification likelihood would be bounded (for $K\geq 2$). Insights to choose lower bounds on the proportions and the covariance matrices are suggested in \citet[Section 5.1]{Bau09}. The upper bound on the covariance matrices and the compactness condition on the means, although not necessary in the standard likelihood framework, do not seem to be avoidable here (see \sub{icl.con.bra}). This is a consequence of the behavior of the entropy term as a component goes to zero. 

The following theorem, which is directly adapted from \citet[Section 5.2]{Vaa98}, gives sufficient conditions for the consistency of a minimum contrast estimator $\hat{\theta}$.  We write $\forall \theta\in\Theta, \forall \widetilde{\Theta}\subset\Theta, d(\theta,\widetilde{\Theta}) = \underset{\widetilde\theta\in\widetilde\Theta}{\inf} \|\theta-\widetilde\theta\|_\infty$.
\begin{The}\label{th.icl.conest} Let $\Theta\subset\mathbb{R}^D$ and $\gamma:\Theta\times\mathbb{R}^d\longrightarrow \mathbb{R}$. Assume:
\begin{align*}\tag{A1}\label{icl.a1}
\exists& \theta^0\in\Theta \text{ such that } \E{f^\wp}{\gamma(\theta^0)}=\min_{\theta\in\Theta}\E{f^\wp}{\gamma(\theta)}\text{($\Leftrightarrow \Theta^0$ is not empty)}
\end{align*}
\begin{equation}\tag{A2}\label{icl.a2}
\forall \varepsilon > 0,\ \inf_{\{\theta: d(\theta,\Theta^0)>\varepsilon \}} \E{f^\wp}{\gamma(\theta)} > \E{f^\wp}{\gamma(\theta^0)}
\end{equation} 
\begin{equation}\tag{A3}\label{icl.a3}
\sup_{\theta\in\Theta}{\Bigl| \gamma_n(\theta)-\E{f^\wp}{\gamma(\theta)} \Bigr|} \xrightarrow{\ \mathbb{P}\ } 0
\end{equation}
Define $\forall n,\hat\theta=\hat\theta(X_1,\dots,X_n) \in \Theta\text{ such that }\gamma_n(\hat\theta) \leq \gamma_n(\theta^0) + \pop(1).$\\
Then $d(\hat\theta,\Theta^0) \xrightarrow[n\rightarrow\infty]{\mathbb{P}} 0.$
\end{The}

The strong consistency holds if \eqref{icl.a3} is replaced by an almost sure convergence (this is the case under the conditions we are to define) and if the inequality in the definition of $\hat\theta$ holds almost surely.

Assumption \eqref{icl.a1} is the least that can be expected. It is guaranteed if the parameter space is compact.

Assumption \eqref{icl.a2} holds, too, under this compactness assumption: since $\theta\in\Theta_K \mapsto \E{f^\wp}{\gamma(\theta)}$ reaches its minimum value on the compact $\Theta_K \backslash \{\theta\in\Theta_K : d(\theta,\Theta_K^0)>\varepsilon\}$, it is necessarily strictly greater than $\E{f^\wp}{\gamma(\theta^0)}$.

Assumption \eqref{icl.a3} is a bit strong but it will be guaranteed under the compactness assumption through bracketing entropy arguments in \sub{icl.con.bra}. 

\begin{SkeProo} The assumptions guarantee a convenient situation. With great probability as $n$ grows, from \eqref{icl.a3}, $\gamma_n(\theta)$ is uniformly close to $\E{f^\wp}{\gamma(\theta)}$: this holds for $\hat\theta$ and $\theta^0$. Then, from the definition of $\hat\theta$, $E_{f^\wp}\bigl[\gamma(\hat\theta)\bigr]$ cannot be much larger than $\E{f^\wp}{\gamma(\theta^0)}$ which reaches the minimal value. By \eqref{icl.a2}, this implies that $\hat\theta$ cannot be far from $\Theta^0$. 

\end{SkeProo}

Let us apply Theorem~\ref{th.icl.conest} to Gaussian mixtures, with $\gamma=-\Lcc$ and $\Theta=\Theta_K$. The two following hypotheses will be involved:
\begin{gather}
\|M\|_r < \infty \text{ with } M(x) = \sup_{\theta\in\Theta} \left| \gamma(\theta;x) \right| < \infty\ f^\wp d\lambda \text{-a.e.} \tag{\HM{}}\label{HM}\\
\|M'\|_r < \infty \text{ with } M'(x) = \sup_{\theta\in\Theta} {\biggl\| \left( \frac{\partial \gamma}{\partial \theta} \right)_{(\theta;x)} \biggr\|_\infty}  <  \infty\ f^\wp d\lambda \text{-a.e.} \tag{\HMp{}}\label{HM'}
\end{gather}

\begin{The}[Weak Consistency of MLccE, compact case]\label{th.icl.conestgmm}
Let $\mathcal{M}$ be a Gaussian mixture model with compact parameter space $\Theta\subset\mathbb{R}^D$.  Let $\Theta^0=\argmin_{\theta\in\Theta}\E{f^\wp}{-\Lcc(\theta)}$. Let $\Theta^\mathcal{O}\subset\mathbb{R}^{D}$ open over which $\Lcc$ is defined, such that $\Theta\subset\Theta^\mathcal{O}$ and assume that \HMp{\Lcc,\Theta^\mathcal{O},1} holds.\\
Let $\MLccE\theta\in\Theta$ be an estimator (almost) maximizing \Lcc:
\begin{equation*}
\forall \theta^0\in\Theta^0, \forall n\in\mathbb{N^*}, -\Lcc(\MLccE\theta)\leq -\Lcc(\theta^0) + \pop(n).
\end{equation*}
Then $d(\MLccE\theta,\Theta^0)\xrightarrow[n\longrightarrow\infty]{\mathbb{P}} 0.$
\end{The}
\HMp{\Lcc,\Theta^\mathcal{O},1} results from lemma \ref{le.icl.Sncomp} and shall be discussed in \sub{icl.con.bra}. 

Under the compactness assumption, $\MLccE\theta$ is then consistent. It is even strongly consistent if it minimizes the empirical contrast almost surely. Let us highlight that it then converges to the set of parameters minimizing the loss function, which has no reason to contain the true distribution ---~except for $K=1$~--- even if the last lies in $\mathcal M$. 

\subsection{Bracketing Entropy and Glivenko-Cantelli Property}\label{icl.con.bra}
Recall a class of functions over $\Rd$ is $\mathbb{P}$-\emph{Glivenko-Cantelli}, with $\mathbb{P}$ a probability measure over \Rd, if it fulfills a uniform law of large numbers for the distribution $\mathbb{P}$.  A sufficient condition for a family $\mathcal{G}$ to be $\mathbb{P}$-Glivenko-Cantelli is that it is not too complex, which can be measured through its \emph{entropy with bracketing}:
\begin{Def}[$L_r(\mathbb{P})$-entropy with bracketing] Let $r\in\mathbb{N}^*$ and $l,u\in L_r(\mathbb{P})$. The bracket $[l,u]$ is the set of all functions $g\in\mathcal{G}$ with $l \leq g \leq u$. $[l,u]$ is an $\varepsilon$-bracket if $\|l-u\|_r\leq\varepsilon$. The bracketing number $N_{[\,]}(\varepsilon,\mathcal{G},L_r(\mathbb{P}))$ is the minimum number of $\varepsilon$-brackets needed to cover $\mathcal{G}$. The entropy with bracketing $\mathcal{E}_{[\,]}(\varepsilon,\mathcal{G},L_r(\mathbb{P}))$ of $\mathcal{G}$ with respect to $\mathbb{P}$ is the logarithm of $N_{[\,]}(\varepsilon,\mathcal{G},L_r(\mathbb{P}))$.
\end{Def}
It is quite natural that the behavior of all functions lying inside an $\varepsilon$-bracket can be uniformly controlled by the behavior of the extrema of the bracket. If those endpoints belong to $L_1(\mathbb{P})$, they fulfill a law of large numbers, and if the number of them needed to cover $\mathcal{G}$ is finite, then this is no surprise that $\mathcal{G}$ can be proved to fulfill a uniform law of large numbers:
\begin{The}
\begin{sloppypar}
Every class $\mathcal{G}$ of measurable functions such that $\mathcal{E}_{[\,]}(\varepsilon,\mathcal{G},L_1(\mathbb{P}))<\infty$ for every $\varepsilon>0$ is $\mathbb{P}$-Glivenko-Cantelli.
\end{sloppypar}
\end{The}

The reader is referred to \citet[Chapter 19]{Vaa98} for accurate definitions and a proof of this result. This is a generalization of the usual Glivenko-Cantelli theorem. We shall prove that the class of functions $\{\gamma(\p\theta) : \theta\in\Theta_K\}$ has finite $\varepsilon$-bracketing entropy for any $\varepsilon>0$ and the assumption \eqref{icl.a3} will be ensured.

From now on, since $\Theta$ is typically assumed to be compact, it is assumed that $\Theta\subset\Theta^\mathcal{O}\subset\mathbb{R}^D$ with $\Theta^\mathcal{O}$ open over which $\gamma$ is defined and $C^1$ for $f^\wp d\lambda$-almost all $x$. This is no problem for Gaussian mixture models with \Lcc (or the standard likelihood by the way), for example with the general or diagonal model. But this requires (with the \Lcc contrast) the proportions to be positive. Actually, this could be avoided here, but we will need this assumption for the definition of $M'$ (Hypothesis \ref{HM'}). As already mentioned, components going to zero must be avoided. For the same technical reason, we have to assume the mean parameters to be bounded. 

Lemma \ref{le.icl.Snconv} guarantees that the bracketing entropy of $\{\gamma(\p\theta) : \theta\in\Theta\}$ is finite for any $\varepsilon$, if $\Theta$ is convex and bounded. The assumption about the differential of the contrast is not a difficulty in our framework, provided that non-zero lower bounds over $\Theta$ on the proportions and the covariance matrices are imposed. The lemma is written for any $\tilde\Theta$ bounded and included in $\Theta$ (which is not assumed to be bounded itself) since it will be applied locally around $\theta^0$ in the \sub{icl.sel}.
 
For any bounded $\widetilde\Theta\subset\mathbb{R}^D$, $\diam \widetilde\Theta = \sup \bigl\{ \|\theta_1-\theta_2\|_\infty : \theta_1,\theta_2\in\widetilde\Theta \bigr\}$.

\begin{Le}[Bracketing Entropy, Convex Case]\label{le.icl.Snconv}
Let $r\in\mathbb{N}^*$, $D\in\mathbb{N}^*$ and $\Theta\subset\mathbb{R}^{D}$ assumed to be convex. Let $\Theta^\mathcal{O} \subset \mathbb{R}^D$ open such that $\Theta\subset\Theta^\mathcal{O}$ and $\gamma:\Theta^\mathcal{O}\times\mathbb{R}^d\longrightarrow \mathbb{R}$. $\theta\in\Theta^\mathcal{O}\longmapsto\gamma(\theta;x)$ is assumed to be $C^1$ for $f^\wp d\lambda$-almost all $x$. Assume that \HMp{\gamma,\Theta,r} holds.
Then 
\begin{equation*}
\forall \widetilde\Theta\subset\Theta, \forall \varepsilon>0,\  N_{[\,]}(\varepsilon,\{\gamma(\theta):\theta\in\widetilde\Theta\},\|\cdot\|_r) \leq \left( \frac{ \|M'\|_r\, \diam \widetilde\Theta}{\varepsilon} \right)^{D}\vee 1.
\end{equation*}
\end{Le}
Remark that $\Theta$ does not have to be compact. 
Its proof is a calculation which relies on the mean value theorem, hence the convexity assumption. The natural parameter space of diagonal Gaussian mixture models, with equal volumes (if $d>1$) or not, for instance, is convex (see Examples \ref{ex.diag.conv} and \ref{ex.diag.convEq}, p.~\pageref{ex.diag.conv}). General mixture models have a convex natural parameter space, too, since the set of definite positive matrices is convex. However, 
there is no reason that the parameter space $\Theta$ should be convex in general. 

Lemma \ref{le.icl.Snconv} can then be generalized at the price of assuming $\Theta$ to be compact, and included in an open set $\Theta^\mathcal O$ such that \HMp{\gamma,\Theta^\mathcal O,r} holds. This is no difficulty for the mixture models we consider, under the same lower bounds constraints as before (since $\Theta^\mathcal{O}$ itself can be chosen to be included in a compact subset of the set of possible parameters). The entropy is then increased by a multiplying factor $Q$, which only depends on $\Theta$ and roughly measures its ``nonconvexity''. Since the exponential behavior of the entropy with respect to $\varepsilon$ is of concern, this does not make the result really weaker. 

\begin{Le}[Bracketing Entropy, Compact Case]\label{le.icl.Sncomp}
Let $r\in\mathbb{N}^*$, $D\in\mathbb{N}^*$ and $\Theta\subset\mathbb{R}^{D}$ assumed to be compact. Let $\Theta^\mathcal{O} \subset \mathbb{R}^D$ open such that $\Theta\subset\Theta^\mathcal{O}$ and $\gamma:\Theta^\mathcal{O}\times\mathbb{R}^d\longrightarrow \mathbb{R}$. $\theta\in\Theta^\mathcal{O}\longmapsto\gamma(\theta;x)$ is assumed to be $C^1$ for $f^\wp d\lambda$-almost all $x$. Assume that \HMp{\gamma,\Theta^{\mathcal O},r} holds.

Then 
\begin{multline*}
\exists Q\in\mathbb{N}^*, \forall \widetilde\Theta\subset\Theta, \forall \varepsilon>0,\\ N_{[\,]}(\varepsilon,\{\gamma(\theta):\theta\in\widetilde\Theta\},\|\cdot\|_r) \leq Q \left( \frac{ \|M'\|_r \, \diam \widetilde\Theta}{\varepsilon} \right)^D\vee 1.
\end{multline*}
$Q$ is a constant which depends on the geometry of $\Theta$ ($Q=1$ if $\Theta$ is convex). 
\end{Le} 
This lemma is proved by applying Lemma \ref{le.icl.Snconv} since $\Theta$ is still locally convex. Since it is compact, it can be covered with a finite number $Q$ of open balls, which are convex. Lemma \ref{le.icl.Snconv} then applies to the convex hull of the intersection of $\Theta$ with each one of them. The supremum of $M'$ is taken over $\Theta^\mathcal{O}$ ---~instead of $\Theta$~--- to make sure that the assumptions of Lemma \ref{le.icl.Snconv} are fulfilled over those entire balls, which may not be included in $\Theta$.

The result we need for \sub{icl.sel} is Lemma \ref{le.icl.Snconv2}, obtained from Lemma \ref{le.icl.Snconv} by a slight modification. Since it is applied locally there, the convexity assumption is no problem. A supplementary and strong assumption \HM{\gamma,\Theta,\infty} is made. This is not fulfilled in the general Gaussian mixtures framework. A sufficient condition is that the support of $f^\wp$ is bounded. This is false of course for most usual distributions we may have in mind, but this is a reasonable modeling assumption: most modeled phenomena are bounded. Another sufficient condition to guarantee this assumption is that the contrast is bounded from above. This is actually not the case of the contrast $-\Lcc$, but this can be imposed: replace $-\Lcc$ by $(-\Lcc\wedge C)$ and, provided that $C$ is large enough, this new contrast behaves like $\Lcc$. This is a supplemental difficulty in practice to choose a relevant $C$ value, though. 

\begin{Le}[Bracketing Entropy, Convex Case]\label{le.icl.Snconv2}
Let $r \geq 2$, $D\in\mathbb{N}^*$ and $\Theta\subset\mathbb{R}^{D}$ assumed to be convex. Let $\Theta^\mathcal{O} \subset \mathbb{R}^D$ open such that $\Theta\subset\Theta^\mathcal{O}$ and $\gamma:\mathbb{R}^D\times\Theta^\mathcal{O}\longrightarrow \mathbb{R}$. $\theta\in\Theta^\mathcal{O}\longmapsto\gamma(\theta;x)$ is assumed to be $C^1$ for $f^\wp d\lambda$-almost all $x$. Assume that \HM{\gamma,\Theta,\infty} and \HMp{\gamma,\Theta,2} hold.
Then 
\begin{multline*}
\forall \widetilde\Theta\subset\Theta, \forall \varepsilon>0,\\ N_{[\,]}(\varepsilon,\{\gamma(\theta):\theta\in\widetilde\Theta\},\|\cdot\|_r) \leq \left( \frac{ 2^{r-2}\, \|M\|_\infty^\frac{r-2}{2}\, \|M'\|_2 \, \diam \widetilde\Theta}{\varepsilon^\frac{r}{2}} \right)^{D}\vee 1.
\end{multline*}
\end{Le}

Let us remark that those results are quite general. We are interested here in their application to the conditional classification likelihood, but they hold all the same in the standard likelihood framework. \cite{MauMic09} already provide bracketing entropy results in this framework. Our results cannot be directly compared to theirs since they consider the Hellinger distance. The dependency they get on the parameter space bounds and the variable space dimension $d$ is explicit. This is helpful to derive an oracle inequality. But they could not derive a local control of the entropy, hence an unpleasant logarithm term in the expression of the optimal penalty they get. Their results also suggest the necessity of assuming the contrast to be bounded: see the discussion after the Theorem \ref{th.icl.selgmm}. The results we propose achieve the same rate with respect to $\varepsilon$. They depend on more opaque quantities ($\|M\|_\infty$ and $\|M'\|_2$). This notably implies, from this first step already, the assumption that the contrast is bounded ---~over the true distribution support. However, it could be expected to control those quantities with respect to the parameter space bounds. Moreover, beside their simplicity, they straightforwardly enable to derive a local control of the entropy.

\section{Model Selection}\label{icl.sel}

As illustrated by Example \ref{ex.icl.ex1}, model selection is a crucial step. The number of classes may even be the target of the study. Anyhow, a relevant number of classes must obviously be chosen so as to design a good classification. 

Model selection procedures introduced here are penalized conditional classification likelihood criteria:
\begin{equation*}
\crit(K) = -\Lcc(\MLccE\theta_K) + \pen(K).
\end{equation*}
Most results are stated for a general contrast $\gamma$ and any family of models $\{\mathcal{M}_K\}_{1\leq K \leq K_M}$ and then applied to $-\Lcc$ and the Gaussian mixtures family of models $\{\mathcal{M}_K\}_{1\leq K \leq K_M}$ introduced in \sect{icl.int.gmm}. 

In \sub{icl.sel.cri}, the consistency of such a model selection procedure (``identification'' point of view) is proved for a class of penalties. Sufficient conditions are given in the general Theorem \ref{th.icl.consel}, which is applied to the framework we are interested in in Theorem \ref{th.icl.selgmm}. The heaviest condition of Theorem \ref{th.icl.consel} \eqref{icl.b4} may be guaranteed under regularity and (weak) identifiability assumptions, and is discussed in \sub{icl.sel.suf}. Our approach is inspired from works of \citet{Mas07} and is the first step to reach non-asymptotic results. 

\subsection{Consistent Penalized Criteria}\label{icl.sel.cri}
Assume that $K_0$ exists such that
\begin{equation*}
\text{\raisebox{.5em}{and}} 
\begin{aligned}
\quad\forall K<K_0, \inf_{\theta\in\Theta_{K_0}}\E{f^\wp}{\gamma(\theta)}<\inf_{\theta\in\Theta_K}\E{f^\wp}{\gamma(\theta)}\\
\quad\forall K\geq K_0,\inf_{\theta\in\Theta_{K_0}}\E{f^\wp}{\gamma(\theta)}\leq\inf_{\theta\in\Theta_K}\E{f^\wp}{\gamma(\theta)}
\end{aligned}
\end{equation*}
which means that the bias of the models is stationary from the model $\mathcal{M}_{K_0}$: it is the ``best'' model. Remark that the last property should hold mostly in the mixtures framework, and notably if the models were not constrained, and then were nested. Under this assumption, a model selection procedure is expected to asymptotically recover $K_0$, \ie to be \emph{consistent}. This is an \emph{identification} aim \citep[see][Chapter 1]{McqTsa98}. It would be disastrous to select a model which does not (almost) minimize the bias. And it is besides assumed that the model $\mathcal{M}_{K_0}$ contains all the interesting information (typically, the structure of the classes).

Let us stress that the ``true'' number of components of $f^\wp$ is not directly of concern: it is in particular not assumed that it equals $K_0$, and is not even assumed to be defined ($f^\wp$ does not have to be a Gaussian mixture). $K_0$ is the best choice from the particular point of view introduced by using the \Lcc contrast, which is not density estimation, neither is it identification of the ``true'' number of components.

\begin{The}\label{th.icl.consel}
$\{\Theta_K\}_{1\leq K \leq K_M}$ a collection of models with $\Theta_K\subset\mathbb{R}^{D_K}$ ($D_1\leq\dots\leq D_{K_M}$) and let $\theta_K^0\in\Theta_K^0$, with $\Theta_K^0=\underset{\theta\in\Theta_K}{\argmin}\ \E{f^\wp}{\gamma(\theta)}$. Assume
\begin{equation}\tag{B1}\label{icl.b1}
K_0=\min\argmin_{1\leq K\leq K_M} \E{f^\wp}{\gamma(\Theta_K^0)}
\end{equation}
\begin{equation}\tag{B2}\label{icl.b2}
\begin{split}
\forall K,\ \hat\theta_K \in \Theta_K&\text{ defined such that }\gamma_n(\hat\theta_K) \leq \gamma_n(\theta_K^0) + \pop(1)\\
&\text{ fulfills }\gamma_n(\hat\theta_K) \stackrel{\mathbb{P}}{\longrightarrow} \E{f^\wp}{\gamma(\theta_K^0)}
\end{split}
\end{equation}
\begin{equation}\tag{B3}\label{icl.b3}
\forall K,
\begin{cases} 
\pen(K)>0 \text{ and } \pen(K)=\pop(1)&\text{when }n \rightarrow +\infty\\
n\bigl(\pen(K)-\pen(K')\bigr) \xrightarrow[n \rightarrow +\infty]{\mathbb{P}} \infty&\text{when }K>K' 
\end{cases}
\end{equation}
\begin{equation}\tag{B4}\label{icl.b4}
n\bigl(\gamma_n(\hat\theta_{K_0})-\gamma_n(\hat\theta_K)\bigr)=\gop(1)\text{ for any }K\in\argmin_{1\leq K\leq K_M} \E{f^\wp}{\gamma(\Theta_K^0)}.
\end{equation}
Define $\hat K$ such that $\hat K = \min \underset{1 \leq K \leq K_M}{\argmin}\  \biggl\{\underbrace{\gamma_n(\hat\theta_K) + \pen(K)}_{\crit(K)}\biggr\}.$\\
Then $\mathbb{P} [\hat K \neq K_0] \xrightarrow[n\longrightarrow \infty]{}0.$
\end{The}

\begin{SkeProo}
First prove that $\hat K$ cannot asymptotically ``underestimate'' $K_0$. Suppose $\E{f^\wp}{\gamma(\theta_K^0)}>\E{f^\wp}{\gamma(\theta_{K_0}^0)}$. From \eqref{icl.b2}, $\bigl(\gamma_n(\hat\theta_K)-\gamma_n(\hat\theta_{K_0})\bigr)$ is asymptotically of order $\E{f^\wp}{\gamma(\theta_K^0)}-\E{f^\wp}{\gamma(\theta_{K_0}^0)}>0$. Since the penalty is $\pop(1)$ from \eqref{icl.b3}, $\crit(K_0)<\crit(K)$ asymptotically and $\hat K > K$.

That $\hat K$ does not asymptotically ``overestimate'' $K_0$, involves the heaviest assumption \eqref{icl.b4}. It is more involved since then $\bigl(\E{f^\wp}{\gamma(\theta_K^0)}-\E{f^\wp}{\gamma(\theta_{K_0}^0)}\bigr)$ is zero. The fluctuations of $\bigl(\gamma_n(\hat\theta_K)-\gamma_n(\hat\theta_{K_0})\bigr)$ around zero then have to be evaluated and canceled by a penalty large enough. According to \eqref{icl.b4}, a penalty larger than $\inv{n}$ should suffice. \eqref{icl.b3} guarantees this condition.
\end{SkeProo}

Assumption \eqref{icl.b3} defines the range of possible penalties; Assumption \eqref{icl.b2} is guaranteed under assumption \eqref{icl.a3} of Theorem \ref{th.icl.conest}:
\begin{Le}\label{le.icl.1}

For a fixed $K$, assume \eqref{icl.a3}. 
Then \eqref{icl.b2} holds.
\end{Le}
Indeed, asymptoticaly, minimizing $\theta\mapsto\gamma_n(\theta)$ cannot differ much from minimizing $\theta\mapsto\E{f^\wp}{\gamma(\theta)}$ if they are uniformly close to each other \eqref{icl.a3}. 

Assumption \eqref{icl.b4} is the heaviest assumption. \sub{icl.sel.suf} is devoted to deriving sufficient conditions so that it holds. This will justify the 

\begin{The}\label{th.icl.selgmm}
$(\mathcal{M}_K)_{1\leq K \leq K_M}$ Gaussian mixture models with compact parameter space $\Theta_K$ and  $\Theta_K^0 = \argmin_{\theta\in\Theta_K} \E{f^\wp}{-\Lcc(\theta)}$ for any $K$. Let $K_0 = \min\underset{1\leq K \leq K_M}{\argmin}\ \E{f^\wp}{-\Lcc(\Theta_{K}^0)}$. Assume $\forall K, \forall \theta\in\Theta_K, \forall \theta_{K_0}^0\in\Theta_{K_0}^0$,
\begin{multline}\tag{C1}\label{eq.identifiability}
\E{f^\wp}{-\Lcc(\theta)}=\E{f^\wp}{-\Lcc(\Theta_{K_0}^0)} \\ \Longleftrightarrow -\Lcc(\theta;x)=-\Lcc(\theta_{K_0}^0;x)\quad f^\wp d\lambda-\text{a.e.}
\end{multline}
For any $K$, let $\Theta_K^\mathcal{O} \subset \mathbb{R}^{D_K}$ open over which $\Lcc$ is defined, such that $\Theta_K\subset\Theta_K^\mathcal{O}$; assume that \HM{\Lcc,\Theta_K^{\mathcal O},\infty} and \HMp{\Lcc,\Theta_K^\mathcal{O},2} hold and that $\forall\theta_K^0\in\Theta_K^0,\ I_{\theta_K^0} = \frac{\partial^2}{\partial \theta^2}\bigl( \E{f^\wp}{-\Lcc(\theta)} \bigr)_{|\theta_K^0}$ is nonsingular; let $\MLccE\theta_K\in\Theta_K$ with
\begin{equation*}
-\Lcc(\MLccE\theta_K)\leq-\Lcc(\theta_K^0) + \pop(n).
\end{equation*}
Let $\pen:\{1,\dots,K_M\}\longrightarrow\mathbb{R}^+$ (which may depend on $n$, $\bigl(\Theta_K\bigl)_{1\leq K\leq K_M}$ and the data) such that 
\begin{equation*}
\forall K\in\{1,\dots,K_M\},
\begin{cases} 
\pen(K)>0 \text{ and } \pen(K)=\pop(n)&\text{when }n \rightarrow +\infty\\
\bigl(\pen(K)-\pen(K')\bigr) \xrightarrow[n \rightarrow +\infty]{\mathbb{P}} \infty&\text{for any }K'<K.
\end{cases}
\end{equation*}

Select $\hat K$ such that $\hat K = \min \underset{1\leq K \leq K_M}{\argmin} \bigl\{-\Lcc(\MLccE\theta_K) + \pen(K)\bigr\}.$\\
Then $\mathbb{P}[\hat K \neq K_0] \xrightarrow[n\longrightarrow\infty]{} 0$.
\end{The}

If $\Theta_K$ is convex, $M$ and $M'$ can be defined as suprema over $\Theta_K$ instead of $\Theta_K^\mathcal{O}$ and there is no need to introduce the sets $\Theta_K^\mathcal{O}$. 
The new ``identifiability'' assumption \eqref{eq.identifiability} introduced is reasonable: as expected the label switching phenomenon is no problem here. But it is necessary for the identification point of view to make sense, that a single value of the \emph{contrast function} $x\longmapsto\gamma(\theta;x)$ minimizes the loss. Remark that in the standard likelihood framework, this holds at least if any model contains the sample distribution, since it is the unique Kullback-Leibler divergence minimizer. Obviously, several parameter values, perhaps in different models, may represent it, besides the label switching. We do not know any such result with the $-\Lcc$ contrast and hypothesize that the assumption holds. 

The assumption about the nonsingularity of $I_{\theta^0}$ is unpleasant, since it is hard to be guaranteed. Hopefully, it could be weakened. The result of \cite{Mas07} (Theorem 7.11) which inspires this, and is available in a standard likelihood context, does not require such an assumption since it does not rely on the study of this link between the contrast and the parameters but on a clever choice of the involved distances (Hellinger distances), and on particular properties of the $\log$ function. However, this is a usual assumption \citep[see][or below]{RedWal84}.

\label{icl.sel.bdcon} \cite{Mas07} moreover does not require the contrast (\ie the likelihood) to be bounded, as we have to. Remark however that the application of his Lemma 7.23 to obtain a genuine oracle inequality involves an assumption similar to the boundedness of the contrast. So that it seems reasonable that the assumptions about $M$ and $M'$ (the last is much milder than the former) be necessary. They are typically ensured if either the contrast is bounded or if the support of $f^\wp$ is bounded. 

The conditions about the penalty form are analogous to that of \cite{Nis88} or \cite{Ker00}, which are both derived in the standard maximum likelihood framework. As those of \cite{Ker00}, they can be regarded as generalizing those of \cite{Nis88} when the considered models are Gaussian mixture models. Indeed, \cite{Nis88} considers penalties of the form $c_n D_K$ and proves the model selection procedure to be weakly consistent if $\frac{c_n} n \rightarrow 0$ and $c_n\rightarrow \infty$. Note that \cite{Nis88} assumes the parameter space to be convex. He moreover notably assumes that $\Theta_K^0=\{\theta_K^0\}$ and that the counterpart of $I_{\theta_K^0}$ is nonsingular, together with other regularity assumptions. Those results are not particularly designed for mixture models. Instead, as we do, \cite{Ker00} considers general penalty forms and proves the procedure to be consistent if $\frac{\pen(K)} n \rightarrow 0$, $\pen(K)\rightarrow \infty$ and $\liminf \frac{\pen(K)}{\pen(K')} > 1$ if $K>K'$. These conditions are equivalent to Nishii's if $\pen(K)=c_n D_K$. In a general mixture model framework, she assumes the model family to be well-specified, the same notion of identifiability as we do, and a condition which does not seem to be directly comparable to ours  about $I_{\theta_K^0}$ but which tastes roughly the same. It might be milder. Those assumptions are proved to hold with the standard likelihood contrast for Gaussian mixture models with lower bounded, spherical covariance matrices which are the same for all components, and if the means belong to a compact. Our conditions about the penalty are a little weaker than Keribin's, but they still are quite analogous. Moreover, as compared to those results, we notably have to keep the proportions away from zero. This is necessary because the entropy term must be handled. It does not seem easy to extend the methods used by \cite{Ker00} to our framework. 

The strong version of Theorem \ref{th.icl.selgmm}, which would state the almost sure consistency of $\hat K$ to $K_0$, would then probably involve penalties a little heavier, as \cite{Nis88} and \cite{Ker00} proved in their respective frameworks: both had to assume that $\frac{\pen(K)}{\log\log n}\rightarrow\infty$.

Theorem \ref{th.icl.selgmm} is a direct consequence of Theorem \ref{th.icl.consel}, Lemma \ref{le.icl.1}, Theorem \ref{th.icl.conestgmm}, which can be applied under those assumptions, and of Corollary \ref{cor.icl.nSn} below and the discussion about its assumptions along the lines of \sub{icl.sel.suf}.

\subsection{Sufficient Conditions to Ensure Assumption \eqref{icl.b4}}\label{icl.sel.suf}
Let us introduce the notation $S_n \gamma(\theta) = n\bigl(\gamma_n(\theta) - \E{f^\wp}{\gamma(\theta)}\bigr)$. The main result of this section is Lemma \ref{le.icl.Sn}. Some intermediate results which enable to link Lemma \ref{le.icl.Sn} to Theorem \ref{th.icl.consel} via Assumption \eqref{icl.b4} are stated as corollaries and proved subsequently. Lemma \ref{le.icl.Sn} povides a control of $\sup_{\theta\in\Theta}\frac{S_n\left(\gamma(\theta^0)-\gamma(\theta)\right)}{\|\theta^0-\theta\|_\infty^2 + \beta^2}$ (with respect to $\beta$) and then of $\frac{S_n\left(\gamma(\theta^0)-\gamma(\hat\theta)\right)}{\|\theta^0-\hat\theta\|_\infty^2 + \beta^2}$. With a good choice of $\beta$, and if $S_n\bigl(\gamma(\theta^0)-\gamma(\hat\theta)\bigr)$ can be linked to $\|\theta^0-\hat\theta\|_\infty^2$, it is proved in Corollary \ref{cor.icl.sqrtcon} that it may then be assessed that $n\|\hat\theta-\theta^0\|_\infty^2=\gop(1).$

\begin{sloppypar}
Plugging this last property back into the result of Lemma \ref{le.icl.Sn} yields (Corollary \ref{cor.icl.nSn}) $n\bigl(\gamma_n(\theta^0_K)-\gamma_n(\hat\theta_K)\bigr)=\gop(1)$ for any model $K\in\argmin_{1\leq K\leq K_M} \E{f^\wp}{\gamma(\theta^0_K)}$ and then, under mild identifiability condition, $n\bigl(\gamma_n(\theta^0_{K_0})-\gamma_n(\hat\theta_K)\bigr)=\gop(1),$ which is Assumption \eqref{icl.b4}.
\end{sloppypar}

\begin{Le}\label{le.icl.Sn}Let $D\in\mathbb{N}^*$ and $\Theta\subset\mathbb{R}^{D}$ convex. Let $\Theta^\mathcal{O}\subset\mathbb{R}^D$ open such that $\Theta\subset\Theta^\mathcal{O}$ and $\gamma:\Theta^\mathcal{O}\times\Rd\rightarrow \mathbb{R}$. $\theta\in\Theta^\mathcal{O}\mapsto\gamma(\theta;x)$ is assumed to be $C^1$ over $\Theta^\mathcal{O}$ for $f^\wp d\lambda$-almost all $x$. Let $\theta^0\in\Theta$ such that $\E{f^\wp}{\gamma(\theta^0)}=\underset{\theta\in\Theta}{\inf} \E{f^\wp}{\gamma(\theta)}$. Assume that \HM{\gamma,\Theta,\infty} and \HMp{\gamma,\Theta,2} hold.\\
Then $\exists \alpha>0/ \forall n, \forall \beta>0, \forall \eta>0$, with probability larger than $(1-\exp{-\eta})$,
\begin{multline*}
\sup_{\theta	\in\Theta} \frac{S_n(\gamma(\theta^0)-\gamma(\theta))}{\|\theta^0-\theta\|_\infty^2 + \beta^2} 
\leq \frac{\alpha}{\beta^2} \Bigl(\|M'\|_2 \beta \sqrt{nD}+ \bigl(\|M\|_\infty + \|M'\|_2 \beta\bigr)D\\
+ \|M'\|_2 \sqrt{n\eta} \beta+\|M\|_\infty \eta\Bigr)
\end{multline*}
Note that $\alpha$ is an absolute constant which notably does not depend on $\theta^0$.
\end{Le}

\begin{SkeProo}
The proof relies on results of \cite{Mas07} and on the evaluation of the bracketing entropy of the class of functions at hand. Lemma \ref{le.icl.Snconv2} provides a local control of the entropy and hence, through Theorem 6.8 in \citet{Mas07}, a control of the supremum of $S_n(\gamma(\theta^0)-\gamma(\theta))$ as $\|\theta-\theta^0\|^2_\infty<\sigma$, with respect to $\sigma$. The ``peeling'' Lemma 4.23 in \citet{Mas07} then enables to take advantage of this local control to derive a fine global control of $\sup_{\theta\in\Theta} \frac{S_n(\gamma(\theta^0)-\gamma(\theta))}{\|\theta-\theta^0\|^2 + \beta^2}$, for any $\beta$. This control in expectation, which can be derived conditionally to any event A, yields a control in probability thanks to Lemma 2.4 in \citet{Mas07}, which can be thought of as an application of Markov's inequality.
\end{SkeProo}

\begin{Cor}\label{cor.icl.sqrtcon}Same assumptions as Lemma~\ref{le.icl.Sn}, but the convexity of $\Theta$. Besides assume that $I_{\theta^0}=\dd{\E{f^\wp}{\gamma(\theta)}}{\theta}{\theta^0}$ is nonsingular. Let $(\hat \theta_n)_{n\geq 1}$ such that $\hat\theta_n\in\Theta$, $\gamma_n(\hat\theta_n)\leq\gamma_n(\theta^0) + \gop(\inv{n})$ and $\hat\theta_n \xrightarrow[n\rightarrow\infty]{\mathbb{P}} \theta^0.$
Then 
\begin{equation*}
n\|\hat\theta_n - \theta^0\|_\infty^2 = \gop(1).
\end{equation*}
The constant involved in $\gop(1)$ depends on $D$, $\|M\|_\infty$, $\|M'\|_2$ and $I_{\theta^0}$. 
\end{Cor}

This is a direct consequence of Lemma \ref{le.icl.Sn}: it suffices to choose $\beta$ well. The dependency of $\gop(1)$ in $D$, $\|M\|_\infty$, $\|M'\|_2$ and $I_{\theta^0}$ is not a problem since we aim at deriving an asymptotic result: the order of $\|\theta-\theta^0\|^2_\infty$ with respect to $n$  when the model is fixed is of concern. 

The assumption that $I_{\theta^0}$ is nonsingular plays an analogous role as Assumption \eqref{icl.a2} in Theorem \ref{th.icl.conest}: this ensures that $\E{f^\wp}{\gamma(\theta)}$ cannot be close to $\E{f^\wp}{\gamma(\theta^0)}$ if $\theta$ is not close to $\theta^0$. But this stronger assumption is necessary to strengthen the conclusion: the rate of the relation between $\E{f^\wp}{\gamma(\theta)}-\E{f^\wp}{\gamma(\theta^0)}$ and $\|\theta-\theta^0\|$ can then be controlled...

Should this assumption fail, $\exists \tilde\theta\in\Theta/\tilde\theta'I_{\theta^0}\tilde\theta = 0 \Rightarrow \E{f^\wp}{\gamma(\theta^0+\lambda\tilde\theta)} = \E{f^\wp}{\theta^0} + \po(\lambda^2)$ and then there is no hope to have $\alpha>0$ such that $\E{f^\wp}{\gamma(\theta)}-\E{f^\wp}{\gamma(\theta^0)}>\alpha \|\theta-\theta^0\|^2$: this approach cannot be applied without this ---~admittedly unpleasant~--- assumption. Perhaps an other approach (with distances not involving the parameters but directly the contrast values) might enable to avoid it, as \cite{Mas07} did in the likelihood framework.

\begin{Cor}\label{cor.icl.nSn} Let $(\Theta_K)_{1\leq K \leq K_M}$ be models with, for any $K$, $\Theta_K\subset\mathbb{R}^{D_K}$. Assume that $D_1\leq\dots\leq D_{K_M}$. For any $K$, assume there exists an open set $\Theta_K^\mathcal{O}\subset\mathbb{R}^{D_K}$ such that $\Theta_K\subset\Theta_K^\mathcal{O}$ and such that with $\Theta^\mathcal{O}=\Theta_1^\mathcal{O}\cup\dots\cup\Theta_{K_M}^\mathcal{O}$, $\gamma:\Theta^\mathcal{O}\times\Rd\longrightarrow \mathbb{R}$ is defined and $C^1$ for $f^\wp d\lambda$-almost all $x$. Assume that \HM{\gamma,\Theta^{\mathcal O},\infty} and \HMp{\gamma,\Theta^{\mathcal O},2} hold. Let, for any $K$, $\Theta_K^0=\argmin_{\theta\in\Theta_K} \E{f^\wp}{\gamma(\theta)}$ and $\theta_K^0\in\Theta_K^0$.\\
Let $K_0=\min\argmin_{1\leq K\leq K_M} \E{f^\wp}{\gamma(\Theta_K^0)}$ and assume $\forall K, \forall\theta\in\Theta_K,$
\begin{equation*}
\E{f^\wp}{\gamma(\theta)}=\E{f^\wp}{\gamma(\theta_{K_0}^0)} \Longleftrightarrow \gamma(\theta)=\gamma(\theta_{K_0}^0) \quad f^\wp d\lambda-\text{a.e.}
\end{equation*}
Let $\mathcal{K}=\Bigl\{K\in\{1,\dots,K_M\} : \E{f^\wp}{\gamma(\theta_K^0)}=\E{f^\wp}{\gamma(\theta_{K_0}^0)}\Bigr\}$.\\
For any $K\in\mathcal{K}$, let $\hat\theta_K\in\Theta_K$ such that 
\begin{equation*}
\gamma_n(\hat\theta_K)\leq\gamma_n(\theta_K^0) + \gop\Bigl(\inv{n}\Bigr)\qquad\text{and}\qquad\hat\theta_K \xrightarrow[n\rightarrow \infty]{\mathbb{P}}\theta_K^0.
\end{equation*}
Assume that $I_{\theta_K^0}=\dd{\phantom{\Bigl(}\hspace{-0.6em}\E{f^\wp}{\gamma(\theta)}}{\theta}{\theta_K^0}$ is nonsingular for any $K\in\mathcal{K}$.\\
Then
$
\forall K\in\mathcal{K},\ n\bigl(\gamma_n(\hat\theta_{K_0})-\gamma_n(\hat\theta_K)\bigr) = \gop(1).
$
\end{Cor}

This last corollary states conditions under which assumption \eqref{icl.b4} of Theorem \ref{th.icl.consel} is ensured. 

\subsection{A New Light on ICL}\label{icl.sel.icl}

The previous section suggests links between model selection penalized criteria with the standard likelihood on the one hand and with the conditional classification likelihood we defined on the other hand. Indeed penalties with the same form as those given by \cite{Nis88} or \cite{Ker00} with the standard likelihood are proved to be ``consistent'' in our framework. Therefore, by analogy with the standard likelihood framework, it is expected that penalties proportional to $D_K$ conform an efficiency point of view (think of AIC), and that penalties proportional to $D_K \log n$ are optimal for an identification purpose (think of BIC). This possibility to derive an identification procedure from an\supprok{optimal} efficient procedure by a $\log n$ factor is notified for example by \cite{Arl07}.

Let us then consider by analogy with BIC the penalized criterion
\begin{equation*}
\crit_{\lcc-ICL}(K) = \Lcc(\MLccE\theta_K) - \frac{\log n}{2} D_K.
\end{equation*}
The point is that we almost recover ICL (replace $\MLE\theta_K$ by $\MLccE\theta$ in \eqref{eq.icl.deficl}), which may then be regarded as an approximation of this \lcc-ICL criterion. The corresponding penalty is $\frac{\log n}{2}D_K$, and the derivation of \lcc-ICL illustrates that the entropy should not be considered as a part of the penalty. This notably justifies why ICL does not select the same number of components as BIC or any consistent criterion in the standard likelihood framework, even asymptotically. Actually, it should not be expected to do so. 

When $\MLccE\theta_K$ differs from $\MLE\theta_K$, the former provides more separated clusters. The compromise between the Gaussian component and the cluster viewpoint is achieved with $\MLccE\theta_K$ from the very estimation step. The user is provided a solution which aims at this compromise for each number of classes $K$. However, the number of classes selected through \lcc-ICL differs seldom from the one selected by ICL in simulations \citep[See][Chapter 4]{Bau09}. 

Finally, \lcc-ICL is quite close to ICL and enables to better understand the concepts underlying ICL. ICL remains attractive though, notably because it is easier to implement than \lcc-ICL.

\section{Discussion}\label{icl.dis}
Two families of criteria, in the clustering framework, are distinguished: it is shown that ICL's purpose is of different nature than that of BIC or AIC. The identification theory for the criteria based on the conditional classification likelihood is ---~not surprisingly~--- very similar to the one for the standard likelihood. A major interest of the newly introduced estimator and criteria is to better understand the ICL criterion and the underlying notion of class. This is nor a simple notion of cluster ---~as for example for the k-means procedure~--- neither a pure notion of ``component'' ---~as underlying the MLE/BIC approach~--- but a compromise between both. ICL leads to discovering classes matching a subtle combination of the notions of well separated, compact, clusters, and (Gaussian) mixture components. It then enjoys the flexibility and modeling possibilities of the model-based clustering approach, but does not break the expected notion of cluster. Better understanding of the ICL criterion now means better understanding the newly involved contrast \lcc.

The choice of the involved mixture components must be handled with care in this framework since it leads the cluster shape underlying the study. Several forms of Gaussian mixtures may be involved: for example, spherical and general models may be compared, or models with free proportions may be compared with models with equal proportions. 

Besides it should be further studied how the complexity of the models should be measured when several model kinds are compared. The dimension of the model as a parametric space works for the reported theoretical results. But we are not completely convinced that it is the finest measure of the complexity of Gaussian mixture models. As a matter of fact this simple parametric point of view amounts to considering that all parameters play an analogous role. This is not really natural. 

A further theoretical step would be to drive non-asymptotic results and oracle inequalities. This may give more precise insights about the best penalty shape to use, and then justify the use of the slope heuristics of \citet{BirMas06} (\citealp[see also][]{BauMauMic11} or \citealp{Bau09} for simulations and discussions on this topic).

A practical challenge is to provide efficient optimization algorithms. Some work has been done in this direction already: see \citet{BauCelMar08} and \citet[Section 5.1]{Bau09}. But they need be improved to be more reliable, and above all to run much faster, which would obviously be a condition for a spread practical use of the new contrast. 

A possibility to make this contrast more flexible would be to assign different weights to the log likelihood and the entropy: $\Lcc_\alpha = \alpha\log\Lo + (1-\alpha) \ent$, with $\alpha\in [0;1]$. This would enable to tune how important the assignment confidence is with respect to the Gaussian fit... The difficulty would then be to choose $\alpha$. A first insight which comes in mind is to calibrate $\alpha$ from simulations of situations in which the user knows what solution he expects. 

\section{Proofs}\label{icl.pro}
\begin{proof}[Proof (Max of $h_K$, p.~\pageref{hk})]
\begin{sloppypar}
If $h_K$ reaches a max. value at $(t_1^0,\dots,t_K^0)$ under the constraint $\sum_{k=1}^K t_k^0=1$ then, with $S:(t_1,\dots,t_K)\mapsto\sum_{k=1}^K t_k$, 
\end{sloppypar}
\begin{equation*}
\exists \lambda\in\mathbb{R}/ dh_K(t_1^0,\dots,t_K^0)=\lambda dS(t_1^0,\dots,t_K^0).
\end{equation*}
This is equivalent to $\forall k, \log t_k^0 + 1 = \lambda$. Then, $\forall k,\forall k', t_k^0=t_{k'}^0$ and since $\sum_{k=1}^K t_k^0=1$, this yields $t_k^0=\inv{K}$.
\end{proof}
\begin{proof}[Proof of Theorem \ref{th.icl.conest}]\label{pr.icl.conest}
Let $\varepsilon>0$ and $\eta = \inf_{d(\theta,\Theta^0)>\varepsilon} \E{f^\wp}{\gamma(\theta)} - \E{f^\wp}{\gamma(\theta^0)}>0$ (from assumption \eqref{icl.a2}).
For $n$ large enough and with large probability, from assumption \eqref{icl.a3} and the definition of $\hat\theta$,
\begin{eqnarray*}
\displaystyle \sup_{\theta\in\Theta}\, \bigl|\gamma_n(\theta) - \E{f^\wp}{\gamma(\theta)}\bigr| < \frac{\eta}{3}&\text{and}&
\gamma_n(\hat\theta)\leq \gamma_n(\theta^0) + \frac{\eta}{3}\cdot
\end{eqnarray*}
Then 
\begin{align*}
\mathbb E_{f^\wp}{\bigl[ \gamma(\hat\theta)\bigr]}-\E{f^\wp}{\gamma(\theta^0)}&\leq\mathbb E_{f^\wp}{\bigl[ \gamma(\hat\theta)\bigr]} - \gamma_n(\hat\theta)
+ \gamma_n(\hat\theta) - \gamma_n(\theta^0)\\
&\quad\quad\qquad\qquad\qquad + \gamma_n(\theta^0) - \E{f^\wp}{\gamma(\theta^0)} \\
&< \eta.
\end{align*}
And $d(\hat\theta,\Theta^0)<\varepsilon$ with great probability, as $n$ is large enough.
\end{proof}
\begin{proof}[Proof of Lemma \ref{le.icl.Snconv}]
\supprok{Il y avait des $\Theta_K$ et $D_K$ partout, au lieu de $\Theta$ et $D$ : je corrige. Je remplace les $d$ par des $i$ (d est la dimension des variables). }
Let $\varepsilon>0$, and $\widetilde \Theta \subset \Theta$, with $\widetilde\Theta$ bounded. Let $\widetilde\Theta_\varepsilon$ be a grid in $\Theta$ which ``$\varepsilon$-covers'' $\widetilde\Theta$ in any dimension with step $\varepsilon$. $\widetilde\Theta_\varepsilon$ is for example $\widetilde\Theta_\varepsilon^1 \times \cdots \times \widetilde\Theta_\varepsilon^{D}$ with
\begin{equation*}
\forall i\in\{1,\dots,D\},\widetilde\Theta_\varepsilon^i=\Bigl\{\widetilde\theta_{\min}^i,\widetilde\theta_{\min}^i+\varepsilon,\dots,\widetilde\theta_{\max}^i\Bigr\},
\end{equation*}
where
\begin{equation*}
\forall i\in\{1,\dots,D\},\Bigl\{ \theta^i : \theta\in\widetilde\Theta \Bigr\} \subset \left[\widetilde\theta_{\min}^i-\frac{\varepsilon}{2},\widetilde\theta_{\max}^i+\frac{\varepsilon}{2}\right].
\end{equation*}
This is always possible since $\Theta$ is convex.
For the sake of simplicity, it is assumed without loss of generality, that $\widetilde\Theta_\varepsilon\subset\widetilde\Theta$.
With the $\|\cdot\|_\infty$ norm, the step of the grid $\widetilde\Theta_\varepsilon$ is the same as the step over each dimension, $\varepsilon$:
\begin{equation*}
\forall \tilde\theta\in\widetilde\Theta,\exists \tilde\theta_\varepsilon\in\widetilde\Theta_\varepsilon/\|\tilde\theta-\tilde\theta_\varepsilon\|_\infty\leq\frac{\varepsilon}{2}.
\end{equation*}
And the cardinal of $\widetilde\Theta_\varepsilon$ is at most 
\begin{equation*}
\prod_{i=1}^{D} \frac{(sup_{\theta\in\widetilde\Theta} \theta^i-inf_{\theta\in\widetilde\Theta} \theta^i)}{\varepsilon}\vee 1 \leq \left( \frac{\diam \widetilde\Theta}{\varepsilon} \right)^{D}\vee 1.  
\end{equation*}
Now, let $\theta_1$ and $\theta_2$ in $\Theta$ and $x\in\Rd$.
\begin{align*}
\bigl|\gamma(\theta_1;x)-\gamma(\theta_2;x)\bigr| &\leq \sup_{\theta\in [\theta_1;\theta_2]} \left\|\left(\frac{\partial \gamma}{\partial \theta}\right)_{(\theta;x)}\right\|_\infty \|\theta_1-\theta_2\|_\infty \\
& \leq \underbrace{\sup_{\phantom{[;\theta}\theta\in\Theta\phantom{\theta]}} \left\|\left(\frac{\partial \gamma}{\partial \theta}\right)_{(\theta;x)}\right\|_\infty}_{M'(x)} \|\theta_1-\theta_2\|_\infty,
\end{align*}
since $\Theta$ is convex. Let $\tilde\theta\in\widetilde\Theta$ and choose $\tilde\theta_\varepsilon\in\widetilde\Theta_\varepsilon$ such that $\|\tilde\theta-\tilde\theta_\varepsilon\|_\infty\leq\frac{\varepsilon}{2}$. Then 
\begin{gather*}
\text{\raisebox{-1em}{and}} 
\qquad \qquad \quad \quad \qquad \forall x\in\Rd, \bigl|\gamma(\tilde\theta_\varepsilon;x)-\gamma(\tilde\theta;x)\bigr| \leq M'(x)\frac{\varepsilon}{2} \qquad \qquad \quad \quad \\
\qquad \qquad \qquad \quad \quad \gamma(\tilde\theta_\varepsilon;x)-\frac{\varepsilon}{2}M'(x) \leq \gamma(\tilde\theta;x) \leq \gamma(\tilde\theta_\varepsilon;x)+\frac{\varepsilon}{2}M'(x).\qquad \qquad \qquad \quad 
\end{gather*}
The set of $\varepsilon \|M'\|_r$-brackets (for the $\|\cdot\|_r$-norm) 
\begin{equation*}
\left\{[\gamma(\tilde\theta_\varepsilon)-\frac{\varepsilon}{2}M';\gamma(\tilde\theta_\varepsilon)+\frac{\varepsilon}{2}M'] : \tilde\theta_\varepsilon\in\widetilde\Theta_\varepsilon \right\}
\end{equation*}
then has cardinal at most $\left( \frac{\diam \widetilde\Theta}{\varepsilon} \right)^{D}\vee 1$ and covers $\left\{ \gamma(\tilde\theta):\tilde\theta\in\widetilde\Theta\right\}$.
\end{proof}
\begin{Ex}[Diagonal Gaussian Mixture Model Parameter Space is Convex]\label{ex.diag.conv}
Following \cite{CelGov95}, we write $[p \lambda_k B_k]$ for the model of Gaussian mixtures with diagonal covariance matrices and equal mixing proportions. To keep simple notation, let us consider the case $d=2$ and $K=2$ ($d=1$ or $K=1$ are obviously particular cases!). A natural parametrization of this model (which dimension is $8$) is
\begin{equation*}
\theta\in\mathbb{R}^4\times{\mathbb{R}^{+*}}^4 \stackrel{\varphi}{\longmapsto} \inv{2} \phi\Biggl(\p \begin{pmatrix} \theta_1\\\theta_2 \end{pmatrix},\begin{pmatrix} \theta_5&0\\0&\theta_6 \end{pmatrix}\Biggr) + \inv{2} \phi\Biggl(\p \begin{pmatrix} \theta_3\\\theta_4 \end{pmatrix},\begin{pmatrix} \theta_7&0\\0&\theta_8 \end{pmatrix}\Biggr) 
\end{equation*}
Then $[p \lambda_k B_k] = \varphi(\mathbb{R}^4\times{\mathbb{R}^{+*}}^4)$ and the parameter space $\mathbb{R}^4\times{\mathbb{R}^{+*}}^4$ is convex. \supprok{Remark that with the constraint $\begin{pmatrix} \theta_1\\\theta_2\\\theta_5\\\theta_6 \end{pmatrix}\neq\begin{pmatrix} \theta_3\\\theta_4\\\theta_7\\\theta_8 \end{pmatrix}$, the parameter space $\Theta_K\subset\mathbb{R}^4\times{\mathbb{R}^{+*}}^4$ is not convex anymore. However, it can be, by constraining $\Theta$ even more, so as to avoid label switching, for example by ranking the means and covariances parameters from a component to another in lexicographical order: the parameter space is reduced, convex, and still maps through $\varphi$ onto the entire model.}
\end{Ex}
\begin{Ex}[The Same Model with Equal Volumes is Convex, too...]\label{ex.diag.convEq}
$[p \lambda B_k]$ is the same model as in the previous example, but the covariance matrices determinants have to be equal. With $d=2$ and $K=2$, a natural parametrization of this model of dimension $7$ is
\begin{multline*}
\theta\in\mathbb{R}^4\times{\mathbb{R}^{+*}}^3 \stackrel{\varphi}{\longmapsto} \inv{2} \phi\Biggl(\p \begin{pmatrix} \theta_1\\\theta_2 \end{pmatrix},\sqrt{\theta_7}\begin{pmatrix} \theta_5&0\\0&\inv{\theta_5} \end{pmatrix}\Biggr) \\ + \inv{2} \phi\Biggl(\p \begin{pmatrix} \theta_3\\\theta_4 \end{pmatrix},\sqrt{\theta_7}\begin{pmatrix} \theta_6&0\\0&\inv{\theta_6} \end{pmatrix}\Biggr) 
\end{multline*}
Then $[p \lambda B_k] = \varphi(\mathbb{R}^4\times{\mathbb{R}^{+*}}^3)$ and the parameter space $\mathbb{R}^4\times{\mathbb{R}^{+*}}^3$ is convex. 
\end{Ex}
\begin{proof}[Proof of Lemma \ref{le.icl.Sncomp}]
Let $O_1,\dots,O_Q$ be a finite covering of $\Theta$ consisting of open balls such that $\cup_{q=1}^Q O_q\subset \Theta^\mathcal{O}$. Such a covering always exists since $\Theta$ is assumed to be compact. Remark that 
\begin{equation*}
\Theta = \cup_{q=1}^Q (O_q\cap\Theta) \subset \cup_{q=1}^Q \conv{O_q\cap \Theta}.
\end{equation*}
Now, for any $q$, $\conv{O_q\cap \Theta}$ is convex and $\sup_{\theta\in\conv{O_q\cap \Theta}} \left\|\left(\frac{\partial \gamma}{\partial \theta}\right)_{(\theta;x)}\right\|_\infty \leq M'(x)$ since $\conv{O_q\cap \Theta}\subset O_q\subset\Theta^\mathcal{O}$. Therefore, Lemma \ref{le.icl.Snconv} applies to $O_q\cap\widetilde\Theta \subset \conv{O_q\cap \Theta}$:
\begin{equation*}
N_{[\,]}(\varepsilon,\{\gamma(\theta):\theta\in\widetilde\Theta\cap O_q\},\|\cdot\|_r) \leq \left( \frac{ \|M'\|_r\, \diam \widetilde\Theta}{\varepsilon} \right)^{D}\vee 1.
\end{equation*}
Since $N_{[\,]}(\varepsilon,\{\gamma(\theta):\theta\in\widetilde\Theta\},\|\cdot\|_r)\leq N_{[\,]}(\varepsilon,\cup_{q=1}^Q\{\gamma(\theta):\theta\in\widetilde\Theta\cap O_q\},\|\cdot\|_r)$, the result follows.
\end{proof}

\begin{proof}[Proof of Lemma \ref{le.icl.Snconv2}]
Consider the grid $\widetilde\Theta_\varepsilon$ of the proof of Lemma~\ref{le.icl.Snconv}. Let $\theta_1$ and $\theta_2$ in $\Theta$ and $x\in\Rd$. Since $\Theta$ is convex, 
\begin{align*}
\Bigl|\gamma(\theta_1;x)-\gamma(\theta_2;x)\Bigr|^r &\leq \sup_{\theta\in [\theta_1;\theta_2]} \left\|\left(\frac{\partial \gamma}{\partial \theta}\right)_{(\theta;x)}\right\|_\infty^2 \|\theta_1-\theta_2\|_\infty^2 \\
&\qquad \qquad \qquad \qquad \qquad  \times \bigl(2\sup_{\theta\in \{\theta_1,\theta_2\}}|\gamma(\theta;x)|\bigr)^{r-2} \\
&\leq M'(x)^2\|\theta_1-\theta_2\|_\infty^2 (2\|M\|_\infty)^{r-2} \quad f^\wp d\lambda \text{-a.e.}
\end{align*}
Let $\tilde\theta\in\widetilde\Theta$ and choose $\tilde\theta_\varepsilon\in\widetilde\Theta_\varepsilon$ such that $\|\tilde\theta-\tilde\theta_\varepsilon\|_\infty\leq\frac{\varepsilon}{2}$. Then 
\begin{equation*}
\Bigl|\gamma(\tilde\theta_\varepsilon;x)-\gamma(\tilde\theta;x)\Bigr| \leq M'(x)^\frac{2}{r} \left(\frac{\varepsilon}{2}\right)^\frac{2}{r} (2\|M\|_\infty)^\frac{r-2}{r} \quad f^\wp\text{-a.e.}
\end{equation*}
and the set of brackets 
\begin{sloppypar}
\begin{multline*}
\biggl\{ \Bigl[\gamma(\tilde\theta_\varepsilon;x)-\varepsilon^\frac{2}{r} M'(x)^\frac{2}{r} \|M\|_\infty^\frac{r-2}{r} 2^{1-\frac{4}{r}} ; \gamma(\tilde\theta_\varepsilon;x)+\varepsilon^\frac{2}{r} M'(x)^\frac{2}{r} \|M\|_\infty^\frac{r-2}{r} 2^{1-\frac{4}{r}}\Bigr] \\ : \tilde\theta\in\widetilde\Theta_\varepsilon \biggr\}
\end{multline*}
(of $\|\cdot\|_r$-norm length $(2^{2-\frac{4}{r}}) \|M\|_\infty^\frac{r-2}{r} \|M'\|_2^\frac{2}{r} \varepsilon^\frac{2}{r}$) has cardinal at most $\left( \frac{\diam \widetilde\Theta}{\varepsilon} \right)^{D}\vee 1$ and covers $\left\{ \gamma(\tilde\theta):\tilde\theta\in\widetilde\Theta\right\}$, which yields \supprok[29/08/09]{an analogous result as Lemma \ref{le.icl.Sncomp}, but with an upper bound on $N_{[\,]}$ as stated in }Lemma \ref{le.icl.Snconv2}.
\end{sloppypar}
\vspace{-1.5em}
\end{proof}
\begin{proof}[Proof of Theorem \ref{th.icl.consel}]
Let $\mathcal{K}=\argmin_{1\leq K\leq K_M} \E{f^\wp}{\gamma(\theta_K^0)}$. By assumption, $K_0=\min \mathcal{K}$. 

It is first proved that $\hat K$ does not asymptotically ``underestimate'' $K_0$. Let $K\notin\mathcal{K}$. Let $\varepsilon=\inv{2}\bigl(\E{f^\wp}{\gamma(\theta_K^0)}-\E{f^\wp}{\gamma(\theta_{K_0}^0)}\bigr)>0$. From \eqref{icl.b2} and \eqref{icl.b3} ($\pen(K)=\pop(1)$), with large probability and for $n$ large enough:
\begin{eqnarray*}
\bigl|\gamma_n(\hat\theta_K)-\E{f^\wp}{\gamma(\theta_K^0)}\bigr| \leq \frac{\varepsilon}{3}& 
\bigl|\gamma_n(\hat\theta_{K_0})-\E{f^\wp}{\gamma(\theta_{K_0}^0)}\bigr| \leq \displaystyle \frac{\varepsilon}{3}& \pen(K_0)\leq\frac{\varepsilon}{3}.
\end{eqnarray*}
Then
\begin{multline*}
\crit(K) = \gamma_n(\hat\theta_K) + \pen(K) \geq \E{f^\wp}{\gamma(\theta_K^0)} - \frac{\varepsilon}{3} + 0\\
= \E{f^\wp}{\gamma(\theta_{K_0}^0)} + \frac{5\varepsilon}{3} \geq \underbrace{\gamma_n(\hat\theta_{K_0}) + \pen(K_0)}_{\crit(K_0)} + \varepsilon.
\end{multline*}
Then, with large probability and for $n$ large enough, $\hat K \neq K$.

Let now $K\in\mathcal{K}$ (hence $K>K_0$). This part of the result is more involved than the first one but at this stage, it is not more difficult to derive: all the difficulty is hidden in the strong assumption \eqref{icl.b4}... Indeed, it implies that $\exists V>0,$ such that for $n$ large enough and with large probability,
\begin{equation*}
n\bigl(\gamma_n(\hat\theta_{K_0}) - \gamma_n(\hat\theta_K)\bigr)\leq V.
\end{equation*}
Increase $n$ enough so that $n\bigl(\pen(K)-\pen(K_0)\bigr)>V$ with large probability (which is possible from assumption \eqref{icl.b4}). Then, for $n$ large enough and with large probability,
\begin{equation*}
\crit(K) = \gamma_n(\hat\theta_K) + \pen(K) \geq \gamma_n(\hat\theta_{K_0}) - \frac{V}{n} + \pen(K) > \crit(K_0).
\end{equation*}
And then, with large probability and for $n$ large enough, $\hat K \neq K$.

Finally, since $\mathbb{P}[\hat K\neq K_0] = \sum_{K\notin\mathcal{K}} \mathbb{P}[\hat K = K] + \sum_{K\in\mathcal{K},\ K\neq K_0} \mathbb{P}[\hat K = K]$, the result follows.
\end{proof}

\begin{proof}[Proof of Lemma \ref{le.icl.1}] For any $\varepsilon>0$, with large probability and for $n$ large enough:
\begin{multline*}
\underbrace{\gamma_n(\hat\theta)-\mathbb E_{f^\wp}{\bigl[ \gamma(\hat\theta)\bigr]}}_{\geq -\varepsilon} + \underbrace{\mathbb E_{f^\wp}{\bigl[ \gamma(\hat\theta)\bigr]} - \E{f^\wp}{\gamma(\theta^0)}}_{\geq 0} = \gamma_n(\hat\theta) - \E{f^\wp}{\gamma(\theta^0)} \\
= \underbrace{\gamma_n(\hat\theta)-\gamma_n(\theta^0)}_{\leq\varepsilon} + \underbrace{\gamma_n(\theta^0) - \E{f^\wp}{\gamma(\theta^0)}}_{\leq\varepsilon}.
\end{multline*}
\vspace{-4.1em}\\
\end{proof}
\begin{proof}[Proof of Lemma \ref{le.icl.Sn}]
Actually, the proof as it is written below holds for an at most countable model (because this assumption is necessary for Lemma 4.23 and Theorem 6.8 in \citet{Mas07} to hold). But it can be checked that both those results may be applied to a dense subset of $\{\gamma(\theta) : \theta\in\Theta\}$ containing $\theta^0$ and their respective conclusions generalized to the entire set: choose $\Theta^\text{count}$ a countable dense subset of $\Theta$. Then, for any $\theta\in\Theta$, let $\theta_n\in\Theta^\text{count}\xrightarrow[n\rightarrow \infty]{}\theta$. Then, $\gamma(\theta_n;X) \xrightarrow[n\rightarrow \infty]{a.s.} \gamma(\theta;X)$. Now, whatever $g:\mathbb{R}^D\times\bigl({\Rd}\bigr)^n\rightarrow \mathbb{R}$ such that $\theta\in\mathbb{R}^D\mapsto g(\theta,{\bf X})$ continue a.s., $\sup_{\theta\in\Theta} g(\theta;{\bf X}) = \sup_{\theta\in\Theta^\text{count}} g(\theta;{\bf X})$ a.s. Hence, $\E{f^\wp}{\sup_{\theta\in\Theta} g(\theta;{\bf X})} = \E{f^\wp}{\sup_{\theta\in\Theta^\text{count}} g(\theta;{\bf X})}$. Remark however that the models which are actually considered are discrete, because of the computation limitations.

Let us introduce the centered empirical process 
\begin{equation*}
S_n \gamma(\theta) = n\gamma_n(\theta) - n\E{f^\wp}{\gamma(\theta;X)}.
\end{equation*}
Here and hereafter, $\alpha$ stands for a generic absolute constant, which may differ from a line to an other. 
Let $\theta^0\in\Theta$ such that $\E{f^\wp}{\gamma(\theta^0)}=\underset{\theta\in\Theta}{\inf} \E{f^\wp}{\gamma(\theta)}$. Let us define 
\begin{equation*}
\forall \sigma>0, \Theta(\sigma)=\{\theta\in\Theta : \|\theta-\theta^0\|_\infty \leq \sigma \}.
\end{equation*} 
On the one hand, for all $r\in\mathbb{N}^*\backslash\{1\},$
\begin{equation*}
\forall \theta\in\Theta(\sigma),
\left|\gamma(\theta^0;x)-\gamma(\theta;x)\right|^r \leq M'(x)^2 \|\theta^0-\theta\|_\infty^2 (2M(x))^{r-2} 
\end{equation*}
since $\Theta(\sigma)\subset\Theta$ is convex. And thus,
\begin{equation}\label{eq.icl.lip}
\begin{split}
\forall \theta\in\Theta(\sigma),\E{f^\wp}{|\gamma(\theta^0)-\gamma(\theta)|^r} &\leq \|M'\|_2^2\, \|\theta^0-\theta\|_\infty^2 (2\|M\|_\infty)^{r-2} \\
&\leq \frac{r!}{2} (\|M'\|_2 \sigma)^2 \left( \frac{\not{\hspace{-.2em}2}\|M\|_\infty}{\not{\hspace{-.2em}2}} \right)^{r-2}.
\end{split}
\end{equation}
On the other hand, from Lemma \ref{le.icl.Snconv2}, for any $r\in\mathbb{N}^*\backslash\{1\}$, for any $\delta>0$, there exists $C_\delta$ a set of brackets which cover $\{(\gamma(\theta^0)-\gamma(\theta)) : \theta\in\Theta(\sigma)\}$ (deduced from a set of brackets which cover $\{\gamma(\theta) : \theta\in\Theta(\sigma)\}$...) such that:
\begin{equation*}
\forall r\in\mathbb{N}^*\backslash\{1\}, \forall [g_l,g_u]\in C_\delta, \|g_u-g_l\|_r \leq \left(\frac{r!}{2}\right)^\inv{r} \delta^\frac{2}{r} \left( \frac{4 \|M\|_\infty}{3} \right)^\frac{r-2}{r}
\end{equation*}
and such that, writing $\exp{H(\delta,\Theta(\sigma))}$ the minimal cardinal of such a $C_\delta$, 
\begin{equation}\label{eq.icl.Hub}
\exp{H(\delta,\Theta(\sigma))} \leq \left( \frac{\overbrace{\diam \Theta(\sigma)}^{\quad \leq 2\sigma} \|M'\|_2}{\delta} \right)^{D}\vee 1.
\end{equation}
Then, according to Theorem 6.8 in \citet{Mas07}, \\
$\exists \alpha, \forall \varepsilon \in ]0,1], \forall A \text{ measurable such that } \mathbb{P}[A]>0,$
\begin{multline}\label{eq.icl.proth6}
\mathbb{E}^A \left[ \sup_{\theta\in\Theta(\sigma)} S_n\bigl(\gamma(\theta^0)-\gamma(\theta)\bigr) \right] \leq \frac{\alpha}{\varepsilon}\sqrt{n}\int_0^{\varepsilon\|M'\|_2\sigma} \sqrt{H\bigl(u,\Theta(\sigma)\bigr)}\ du \\
\shoveright{+ 2 \Bigl(\frac{4}{3} \|M\|_\infty + \|M'\|_2\sigma\Bigr) H\bigl(\|M'\|_2 \sigma,\Theta(\sigma)\bigr)}\\
+(1+6\varepsilon)\|M'\|_2\sigma\sqrt{2n\log\inv{\mathbb{P}[A]}} + \frac{8}{3}\|M\|_\infty \log\inv{\mathbb{P}[A]}\cdot
\end{multline}
Now, we have
\begin{multline*}
\forall t\in\mathbb{R}^+, \int_0^t \sqrt{\log\inv{u}\vee 0}\ du =\int_0^{t\wedge 1} \sqrt{\log\inv{u}}\ du \\
\hspace{-1em}\leq \sqrt{t\wedge 1} \sqrt{\int_0^{t\wedge 1} \log \inv{u}\ du} = (t\wedge 1) \sqrt{\log \frac{e}{t\wedge 1}},
\end{multline*}
by the Cauchy-Schwarz inequality. Together with \eqref{eq.icl.Hub}, this yields
\begin{equation}\label{eq.icl.pro2}
\begin{split}
\forall t\in\mathbb{R}^+, \int_0^t \sqrt{H(u,\Theta(\sigma))} du &\leq \sqrt{D} \int_0^t \sqrt{\log\frac{2 \|M'\|_2\sigma}{u}\vee 0}\ du\\
&\leq  \sqrt{D} \bigl(t\wedge 2\|M'\|_2\sigma\bigr) \sqrt{\log\frac{e}{\frac{t}{2\|M'\|_2\sigma}\wedge 1}}, 
\end{split}
\end{equation}
after a simple substitution.\\
Next, let us apply Lemma 4.23 in \citet{Mas07}: From \eqref{eq.icl.Hub}, \eqref{eq.icl.proth6} and \eqref{eq.icl.pro2}, 
\begin{equation*}
\forall \sigma>0, \E{f^\wp}{\sup_{\theta\in\Theta(\sigma)} S_n\bigl(\gamma(\theta^0)-\gamma(\theta)\bigr)} \leq \varphi(\sigma),
\end{equation*}
\begin{multline*}
\text{with }\varphi(t)=\frac{\alpha}{\not{\hspace{-.2em}\varepsilon}} \sqrt{n} \sqrt{D} \not{\hspace{-.2em}\varepsilon}\, \|M'\|_2 t \sqrt{\log\frac{2e}{\varepsilon}} + 2\Bigl(\frac{4}{3}\|M\|_\infty+\|M'\|_2 t\Bigr)D \log 2\\
+ (1+6\varepsilon)\|M'\|_2 t \sqrt{2n\log \inv{\mathbb{P}[A]}} + \frac{8}{3} \|M\|_\infty \log \inv{\mathbb{P}[A]}.
\end{multline*}
As required for Lemma 4.23 in \citet{Mas07} to hold, $\frac{\varphi(t)}{t}$ is nonincreasing. It follows
\begin{equation*}
\forall \beta>0, \mathbb{E}^A\left[ \sup_{\theta\in\Theta} \frac{S_n\bigl(\gamma(\theta^0)-\gamma(\theta)\bigr)}{\|\theta^0-\theta\|_\infty + \beta^2} \right] \leq 4 \beta^{-2} \varphi(\beta).
\end{equation*}
We then choose $\varepsilon=1$ and apply Lemma 2.4 in \citet{Mas07}: for any $\eta>0$ and any $\beta>0$, with probability larger than $1-\exp{-\eta}$,
\begin{multline*}
\sup_{\theta\in\Theta} \frac{S_n\bigl(\gamma(\theta^0)-\gamma(\theta)\bigr)}{\|\theta^0-\theta\|_\infty^2 + \beta^2} \leq \frac{\alpha}{\beta^2} \biggl(\sqrt{nD} \|M'\|_2 \beta \sqrt{\log 2e}\\ + \bigl(\|M\|_\infty + \|M'\|_2 \beta\bigr)D \log 2
+ \|M'\|_2 \beta \sqrt{n\eta} + \|M\|_\infty \eta\biggr).
\end{multline*}
\vspace{-3.5em}\\
\end{proof}

\begin{proof}[Proof of Corollary \ref{cor.icl.sqrtcon}]
Let $\varepsilon>0$ such that $B(\theta^0,\varepsilon)\subset\Theta^\mathcal{O}$. Then, since $\hat\theta_n\xrightarrow{\mathbb{P}} \theta^0$, there exists $n_0\in\mathbb{N}^*$ such that, with large probability, for $n\geq n_0$, $\hat\theta_n\in B(\theta^0,\varepsilon)$. Now, $B(\theta^0,\varepsilon)$ is convex and the assumptions of the corollary guarantee that Lemma \ref{le.icl.Sn} applies. Let us apply it to $\hat\theta_n$: 
$\forall n\geq n_0, \forall \beta>0, \text{ with great probability as $\eta$ is large},$
\begin{multline}\label{eq.icl.pro3}
\frac{S_n\bigl(\gamma(\theta^0)-\gamma(\hat\theta_n)\bigr)}{\|\theta^0-\hat\theta_n\|_\infty^2 + \beta^2} \leq \frac{\alpha}{\beta^2} \bigg(\sqrt{nD} \|M'\|_2 \beta + \bigl(\|M\|_\infty + \|M'\|_2 \beta\bigr)D \\
+ \|M'\|_2 \beta \sqrt{n\eta} + \|M\|_\infty \eta\biggr). 
\end{multline}
But since $I_{\theta^0}$ is supposed to be nonsingular, $\forall \theta\in B(\theta^0,\varepsilon)$, 
\begin{equation*}
\begin{split}
\E{f^\wp}{\theta}-\E{f^\wp}{\theta^0} = (\theta-\theta^0)' I_{\theta^0} (\theta-\theta^0) + r(\|\theta-\theta^0\|_\infty)\|\theta-\theta^0\|_\infty^2 \\
\geq \bigl(2\alpha' + r(\|\theta-\theta^0\|_\infty) \bigr) \|\theta-\theta^0\|_\infty^2
\end{split}
\end{equation*}
where $\alpha'>0$ depends on $I_{\theta^0}$ and $r:\mathbb{R}^+\longrightarrow\mathbb{R}$ fulfills $r(x)\xrightarrow[x\rightarrow 0]{}0$. 
Then, for $\|\theta-\theta^0\|_\infty$ small enough ($\varepsilon$ may be decreased...), 
\begin{equation}\label{eq.icl.pro4}
\forall \theta\in B(\theta^0,\varepsilon), \E{f^\wp}{\theta}-\E{f^\wp}{\theta^0}\geq \alpha' \|\theta-\theta^0\|_\infty^2.
\end{equation}
\vspace{-2em}
\begin{align*}
\tag*{Since } S_n\bigl(\gamma(\theta^0)-\gamma(\hat\theta_n)\bigr) &= n\bigl(\gamma_n(\theta^0) - \gamma_n(\hat\theta_n)\bigr) + n \mathbb E_{f^\wp} \bigl[\gamma(\hat\theta_n) - \gamma(\theta^0) \bigr] \\
&\geq \gop(1) + n \mathbb E_{f^\wp} \bigl[ \gamma(\hat\theta_n) - \gamma(\theta^0) \bigr] ,
\end{align*}
\eqref{eq.icl.pro3} together with \eqref{eq.icl.pro4} leads (with great probability) to 
\begin{equation*}
\begin{split}
n\|\hat\theta_n - \theta^0\|_\infty^2 \leq \frac{\|M'\|_2 (\sqrt{n D} + \sqrt{\eta n} + D)\beta + \|M\|_\infty (D + \eta) + \gop(1)}{\frac{\alpha'}{\alpha} - \inv{n \beta^2}\left( \|M'\|_2 (\sqrt{n D} + \sqrt{\eta n} + D)\beta + \|M\|_\infty (D + \eta) \right)},
\end{split}
\end{equation*}
as soon as the denominator of the right-hand side is positive.
It then suffices to choose $\beta$ such that this condition is fulfilled and such that the right-hand side is upper-bounded by a quantity which does not depend on $n$ to get the result. Let us try $\beta=\frac{\beta_0}{\sqrt{n}}$ with $\beta_0$ independent of $n$:
\begin{equation*}
n\|\hat\theta_n - \theta^0\|_\infty^2 \leq \frac{\|M'\|_2 (\sqrt{D} + \sqrt{\eta} + D)\beta_0 + \|M\|_\infty (D + \eta) + \gop(1)}{\frac{\alpha'}{\alpha} - \inv{\beta_0^2}\left( \|M'\|_2 (\sqrt{D} + \sqrt{\eta}+ D)\beta_0 + \|M\|_\infty (D + \eta) \right)} \cdot
\end{equation*}
This only holds if the denominator is positive. Choose $\beta_0$ large enough so as to guarantee this, which is always possible. \supprok[30/08/09]{Increase moreover $n_0$ so as to guarantee that, with large probability, $\forall n\geq n_0, r(n\|\hat\theta_n-\theta^0\|_\infty^2)\leq \inv{2}$. This is possible since $\hat\theta_n\xrightarrow{\mathbb{P}}\theta^0$. }The result follows: with large probability and for $n$ larger than $n_0$, we have $n\|\hat\theta_n - \theta^0\|_\infty^2 = C \gop(1)$ 
with $C$ depending on $D$, $\|M\|_\infty$, $\|M'\|_2$, $I_{\theta^0}$ and $\eta$.
\end{proof}
\begin{proof}[Proof of Corollary \ref{cor.icl.nSn}]
This is a direct application of Corollary \ref{cor.icl.sqrtcon}. Let $K\in\mathcal{K}$:
$\E{f^\wp}{\gamma(\theta_K^0)}=\E{f^\wp}{\gamma(\theta_{K_0}^0)}$. $\Theta_K$ can be assumed to be convex: if it is not, $\hat \theta_K$ lies in $B(\theta^0_{K_0},\varepsilon)\subset\Theta^\mathcal{O}$ with large probability for large $n$ and $\Theta_K$ may be replaced by $B(\theta^0_{K_0},\varepsilon)$. According to Lemma \ref{le.icl.Sn}, with probability larger than $(1-\exp{-\eta})$ for $n$ large, with 
$\beta=\frac{\beta_0}{\sqrt{n}}$ for any $\beta_0>0$:
\vspace{-1.7em}
\begin{multline*}
S_n\bigl(\gamma(\theta_K^0)-\gamma(\hat\theta_K)\bigr) \leq  \alpha \frac{n\|\theta_K^0-\hat\theta_K\|_\infty^2 + \beta_0^2}{\beta_0^2} \biggl( \|M'\|_2 \Big(\sqrt{D_K} + \sqrt{\eta} + \overbrace{\frac{D_K}{\sqrt n}}^{\ \leq D_K}\Bigr)\beta_0 \\ + \|M\|_\infty (D_K + \eta)\biggr).
\end{multline*}
But, according to Corollary \ref{cor.icl.sqrtcon}, $n\|\theta_K^0-\hat\theta_K\|_\infty^2=\gop(1)$. Moreover, by definition,
\vspace{-0.3em}
\begin{equation*}
S_n\bigl(\gamma(\theta_K^0)-\gamma(\hat\theta_K)\bigr) = n\bigl(\gamma_n(\theta_K^0)-\gamma_n(\hat\theta_K)\bigr) + n\Bigl(\underbrace{\mathbb E_{f^\wp} \bigl[\gamma(\hat\theta_K)\bigr] -\E{f^\wp}{\gamma(\theta_K^0)}}_{\geq 0}\Bigr)\\
\end{equation*}
\vspace{-1em} \\
Thus, $
n\bigl(\gamma_n(\theta_K^0)-\gamma_n(\hat\theta_K)\bigr) = \gop(1).	
$
This holds for any $K\in\mathcal{K}$ and then in particular for $K_0$ and $K$. Besides, $\gamma_n(\theta_K^0)=\gamma_n(\theta_{K_0}^0)$ since, by assumption, $\gamma(\theta_K^0)=\gamma(\theta_{K_0}^0)$ $f^\wp d\lambda$-a.e. Hence $
n\bigl(\gamma_n(\hat\theta_{K_0})-\gamma_n(\hat\theta_K)\bigr) = \gop(1).	
$
\end{proof}


\bibliographystyle{apalike} 
\bibliography{/Users/JP/Maths/jp.bib} 

\paragraph{Acknowledgements} I am deeply grateful to G. Celeux (INRIA Saclay \^Ile-de-France and Universit\'e Paris-Sud), J.-M. Marin (Universit\'e Montpellier II) and P. Massart (Universit\'e Paris-Sud) for their essential help. 

\vfill

\flushright 
\noindent \emph{Jean-Patrick Baudry}
\\Universit\'e Pierre et Marie Curie - Paris VI \\
Bo\^ite 158, Tour 15-25, 2\textsuperscript{e} \'etage \\
4 place Jussieu, 75252 Paris Cedex 05\\ 
France.\\ 
Jean-Patrick.Baudry@upmc.fr\\
\url{http://www.lsta.upmc.fr/Baudry}

\end{document}

%% file: aliases.tex
\theoremstyle{plain}
\newtheorem{Def}{Definition}[section]
\newtheorem{The}{Theorem}[section]
\newtheorem{Cor}{Corollary}[section]

\newtheorem{Le}{Lemma}[section]
\newtheorem{Ex}{Example}[section]

\newenvironment{SkeProo}{\begin{proof}[Sketch of proof]}{\end{proof}}

\newcommand{\Rd}{\ensuremath{\mathbb{R}^d}\xspace}
\newcommand{\Sd}{\ensuremath{\mathbb{S}^d_+}\xspace}
\newcommand{\inv}[1]{{\ensuremath{\frac{1}{#1}}}\xspace}

\renewcommand{\exp}[1]{\ensuremath{\mbox{e}^{#1}}\xspace}
\newcommand{\rv}{random variable\xspace}
\newcommand{\prop}{\ensuremath{\pi}}
\newcommand{\Prop}{\ensuremath{\Pi}}
\newcommand{\lo}{\, \Big|\,}
\newcommand{\p}{\,.\,;}

\newcommand{\MLE}[1]{\ensuremath{\widehat{{#1}}^\text{MLE}}\xspace}
\newcommand{\MLccE}[1]{\ensuremath{\widehat{{#1}}^\text{MLccE}}\xspace}
\newcommand{\MAP}[1]{\ensuremath{\widehat{{#1}}^\text{MAP}}\xspace}

\newcommand{\ICL}{\ensuremath{\text{ICL}}\xspace}

\newcommand{\Lc}{\ensuremath{{\text{L}_\text{c}}}\xspace} 
\newcommand{\Lcc}{\ensuremath{{\log \text{L}_\text{cc}}}\xspace} 
\newcommand{\lcc}{\ensuremath{{\text{L}_\text{cc}}}\xspace}  
\newcommand{\Lo}{\ensuremath{\text{L}}\xspace}

\newcommand{\KL}{\ensuremath{d_\text{KL}}\xspace}

\newcommand{\dd}[3]{\ensuremath{\frac{\partial^2}{\partial #2^2}\left(#1\right)_{|#3}}\xspace}
\newcommand{\pop}{o_\mathbb{P}}
\newcommand{\gop}{O_\mathbb{P}}
\newcommand{\po}{o}

\makeatletter \newcommand{\E}[2]{\@ifmtarg{#2}{\mathbb{E}_{#1}}{\mathbb{E}_{#1}\left[ #2 \right]}} \makeatother 
\makeatletter \newcommand{\En}[1]{\@ifmtarg{#1}{\mathbb{E}_{X_1,\dots,X_n}}{\mathbb{E}_{X_1,\dots,X_n} \left[ #1 \right]}} \makeatother
\makeatletter \newcommand{\EX}[1]{\@ifmtarg{#1}{\mathbb{E}_X}{\mathbb{E}_X \left[ #1 \right]}} \makeatother
\makeatletter \newcommand{\EXn}[1]{\@ifmtarg{#1}{\mathbb{E}_X  \mathbb{E}_{X_1,\dots,X_n}}{ \mathbb{E}_X  \mathbb{E}_{X_1,\dots,X_n} \left[ #1 \right]}} \makeatother

\newcommand{\conv}[1]{\ensuremath{\text{conv}(#1)}\xspace}
\newcommand{\ie}{i.e.\@ }
\makeatletter \newcommand{\fig}[2][]{\@ifmtarg{#1}{Figure~\ref{#2}\xspace}{Figure~\ref{#2}~(#1)\xspace}} \makeatother

\newcommand{\sect}[1]{Section~\ref{#1}}
\newcommand{\sub}[1]{Section~\ref{#1}}
\makeatletter \newcommand{\HM}[1]{\@ifmtarg{#1}{\ensuremath{H^M_{\gamma, \Theta, r}}}{\ensuremath{H^M_{#1}}}} \makeatother
\makeatletter \newcommand{\HMp}[1]{\@ifmtarg{#1}{\ensuremath{H^{M'}_{\gamma, \Theta, r}}}{\ensuremath{H^{M'}_{#1}}}} \makeatother

\DeclareMathOperator*{\esssup}{ess\,sup}

\DeclareMathOperator*{\argmin}{argmin}
\DeclareMathOperator*{\argmax}{argmax}
\DeclareMathOperator{\ent}{ENT}
\DeclareMathOperator{\crit}{crit}
\DeclareMathOperator{\pen}{pen}
\DeclareMathOperator{\diam}{\text{diam}}
\DeclareMathOperator{\supp}{\text{supp}}

\newcommand{\supprok}[2][]{}   